\documentclass[12pt,a4paper]{amsart}

\usepackage{amssymb}
\usepackage[margin=3cm]{geometry}
\usepackage{graphicx}
\DeclareGraphicsExtensions{.pdf, .jpg, .png}
\usepackage{color}

\usepackage{enumitem}
\setenumerate{label=(\roman*)}

\renewcommand{\vec}{\mathbf}

\theoremstyle{plain}
\newtheorem{theorem}{Theorem}
\newtheorem{lemma}{Lemma}
\newtheorem{proposition}{Proposition}
\theoremstyle{definition}
\newtheorem{definition}{Definition}
\newtheorem{remark}{Remark}
\newtheorem{example}{Example}

\newcommand{\diag}{\operatorname{diag\,}}
\newcommand{\MCov}{M_\mathrm{Cov}}

\title[Anisotropic polygonal and polyhedral elements]{Anisotropic polygonal and polyhedral discretizations in finite element analysis}
\author{Steffen Wei{\ss}er}
\address{Steffen Wei{\ss}er, Department of Mathematics, Saarland University, 66041 Saarbr\"ucken, Germany}
\email{weisser@num.uni-sb.de}


\usepackage[final,
      pdftitle={Anisotropic polygonal and polyhedral elements},
      pdfauthor={Steffen Weisser},
      pdfkeywords={},
        colorlinks, 
        bookmarksnumbered, 
        bookmarksopen, 
        bookmarksopenlevel=1, 
        pdfstartview=FitH, 
        linkcolor=black, 
        citecolor=black, 
        urlcolor=black, 
        filecolor=black %
]{hyperref}

\begin{document}

\begin{abstract}
New interpolation and quasi-interpolation operators of Cl\'ement- and Scott-Zhang-type are analyzed on anisotropic polygonal and polyhedral meshes. Since no reference element is available, an appropriate linear mapping to a reference configuration plays a crucial role. A priori error estimates are derived respecting the anisotropy of the discretization. Finally, the found estimates are employed to propose an adaptive mesh refinement based on bisection which leads to highly anisotropic and adapted discretizations with general element shapes in two- and three-dimensions.
\end{abstract}

\maketitle

\section{Introduction}
\label{sec:Introduction}
In nowadays computer simulations the use of highly adapted meshes for the treatment of partial differential equations is crucial in order to achieve accurate and efficient results. The adaptive Finite Element Method (FEM) is a well-founded and accepted strategy which reduces the computational cost while improving the accuracy of the approximation. When dealing with highly anisotropic solutions of boundary value problems, it is widely recognized that anisotropic mesh refinements have significant potential for improving the efficiency of the solution process.
Pioneering works for the analysis of Finite Element Methods on anisotropic meshes have been performed by Apel~\cite{Apel1999habil} as well as by Formaggia and Perotto~\cite{FormaggiaPerotto2001,FormaggiaPerotto2003}. 
The meshes usually consist of triangular and quadrilateral elements in two-dimension as well as on tetrahedral and hexahedral elements in three-dimension. 
First results on a posteriori error estimates for driving adaptive mesh refinement with anisotropic elements have been derived by Kunert~\cite{Kunert2000} for triangular and tetrahedral meshes.
For the mesh generation and adaptation different concepts are available which rely on metric-based strategies, see, e.g.,~\cite{Huang2006,AlauzetLoseille2011}, or on splitting of elements, see~\cite{Schneider2013} and the references therein. The anisotropic splitting of classical elements, however, results in certain restrictions why several authors combine this approach with additional strategies like edge swapping, node removal and local node movement. These restrictions come from the limited element shapes and the necessity to remove or handle hanging nodes in the discretization. For three-dimensional elements the situation is even more difficult. In order to relax the admissibility of the meshes one can apply discontinuous Galerkin (DG) methods, see~\cite{GeorgoulisHallHouston2007}, but consequently the conformity of the approximations is lost.

In recent years the attraction of polytopal meshes increased in the discretization of boundary value problems. These meshes consist of polygonal and polyhedral elements in two- and three-dimensions, respectively, and find their applications in polygonal FEM~\cite{SukumarTabarraei2004}, mimetic discretizations~\cite{BeiraoDaVeigaLipnikovManzini2014} as well as in  the BEM-based FEM~\cite{CopelandLangerPusch2009}, where BEM stands for Boundary Element Method, and the Virtual Element Method (VEM)~\cite{BeiraoDaVeigaBrezziCangianiManziniMariniRusso2012}.
One of the promising features is the high flexibility of the element shapes in the discretization. Since the elements may contain an arbitrary number of nodes on their boundary, the notion of ``hanging nodes'' is naturally included in most of the previously mentioned approaches.
A posteriori error estimates have been developed for the BEM-based FEM as well as for the VEM and they have been successfully applied in adaptive mesh refinement strategies, see~\cite{BeiraoDaVeigaManzini2015,BerroneBorio2017,CangianiGeorgoulisPryerSutton2017,Weisser2011,Weisser2017,WeisserWick2017}.

To the best of our knowledge, the polytopal elements have to fulfil some kind of isotropy in all previous publications, i.e., anisotropic elements, which are very thin and elongated, are explicitly excluded from the error analysis.
Since such anisotropic polytopal elements promise a high potential in the accurate resolution of sharp layers in the solutions of boundary value problems due to their enormous flexibility, we develop an appropriate framework in this article.
Geometric information is used in order to characterize the anisotropy of the elements and to give a definition of mesh regularity in a more general sense. In this article, we address the approximation space coming from the BEM-based FEM and the VEM in two- and three-dimensions, but the ideas are also applicable to polygonal FEM~\cite{SukumarTabarraei2004} with harmonic or other generalized barycentric coordinates, see~\cite{Floater2015,JoshiMeyerDeRoseSanocki2007} and the references therein. We study interpolation as well as quasi-interpolation operators and derive a~priori interpolation error estimates that can be applied in the analysis of BEM-based FEM and VEM after the use of C\'ea- or Strang-type lemmata. The derived estimates are further used to steer an anisotropic mesh refinement procedure in which polytopal elements are bisected successively. Numerical experiments demonstrate the flexibility and the potential for highly anisotropic polytopal discretizations.

The article is organized as follows: Section~\ref{sec:Discretization} introduces the approximation space and discusses the regularity as well as the properties of anisotropic polytopal meshes. In Section~\ref{sec:TraceInequalityAndBestApprox}, an anisotropic trace inequality is derived and best approximation results are proved. Quasi-interpolation operators of Cl\'ement- and Scott-Zhang-type are introduced and analyzed in Section~\ref{sec:QuasiInterpolation}. The derived framework is applied to pointwise interpolation in Section~\ref{sec:Interpolation}. Finally, numerical experiments are performed with a new anisotropic mesh refinement strategy in Section~\ref{sec:AdaptiveRefinement} and some conclusions are drawn in Section~\ref{sec:Conclusion}.

\section{Polytopal meshes and discretization}
\label{sec:Discretization}
Let $\Omega\subset\mathbb R^d$ be a bounded polytopal domain in two or three space dimensions and let $\mathcal K_h$ be a decomposition of $\Omega$ into non-overlapping polytopal elements, such that
\[
 \overline\Omega = \bigcup_{K\in\mathcal K_h} \overline K.
\]
For $d=2$, each polygonal element~$K$ consists of nodes and straight edges which are always situated between two nodes. In three space dimensions ($d=3$), the boundary~$\partial K$ of a polyhedral element~$K$ is formed by flat polygonal faces which are again framed by edges and nodes. In the context of polytopal meshes it is explicitly allowed that the dihedral angles of adjacent faces and the angles of neighbouring edges are equal to $\pi$. Thus, the notion of hanging nodes and edges in classical finite element methods is naturally included in polytopal meshes and does not result in any restrictions.

In order to treat the two- and three-dimensional case in the following simultaneously, we denote the $d-1$ dimensional objects, i.e.\ the edges ($d=2$) and the faces ($d=3$), by~$E$ and the set of all of them by~$\mathcal E_h$.
The nodes in the discretization are denoted by $\vec x_i$, $i=1,\ldots,N$, and the indices of the nodes belonging to $K\in\mathcal K_h$ and $E\in\mathcal E_h$ are given by the sets $I(K)$ and $I(E)$, respectively.

We make use of the usual space of square integrable functions $L_2(\omega)$ and the Sobolev Spaces $H^k(\omega)$, $k=1,2$ and denote their norms by $\|\cdot\|_{L_2(\omega)}$ and $\|\cdot\|_{H^k(\omega)}$, respectively, where $\omega\subset\overline\Omega$ is a $d$ or $d-1$ dimensional domain, see~\cite{Adams1975}. The inner product of $L_2(\omega)$ is written as $(\cdot,\cdot)_{L_2(\omega)}$ and the semi-norm in $H^k(\omega)$ as $|\cdot|_{H^k(\Omega)}$.

\subsection{Finite dimensional discretization of the function space}
\label{subsec:DiscretizationSpace}
The discrete function space considered in this publication originates from the BEM-based Finite Element Method~\cite{Weisser2011} and the Virtual Element Method~\cite{BeiraoDaVeigaBrezziCangianiManziniMariniRusso2012}. 
For $d=2$, we have
\[
 V_h = \left\{v\in H^1(\Omega): \Delta v\big|_K = 0 \;\forall K\in\mathcal K_h \mbox{ and } v\big|_E \mbox{ linear } \forall E\in\mathcal E_h\right\}.
\]
In the two-dimensional case the basis functions of $V_h$ are also known as generalized barycentric coordinates under the name harmonic coordinates, see~\cite{JoshiMeyerDeRoseSanocki2007}. This nodal basis can be constructed as
\begin{gather}
 -\Delta\psi_i = 0 \quad\mbox{in } K, \quad\forall K\in\mathcal K_h,\nonumber\\
 \psi_i(\vec x_j) = \delta_{ij} \quad\mbox{for } j=1,\ldots,N,\label{eq:BasisFunc2D}\\
 \psi_i \mbox{ linear on each edge}, \nonumber 
\end{gather}
for $i=1,\ldots,N$. Each basis function~$\psi_i$ is thus the solution of a local boundary value problem over each element~$K\in\mathcal K_h$. For $d=3$ this definition generalizes according to~\cite{RjasanowWeisser2014} to
\[
 V_h = \left\{v\in H^1(\Omega): \Delta v\big|_K = 0 \;\forall K\in\mathcal K_h \mbox{ and } v\big|_E\in V_h(E) \;\forall E\in\mathcal E_h\right\},
\]
where $V_h(E)$ denotes the two-dimensional discretization space over the face~$E$. The nodal basis functions are constructed as in~\eqref{eq:BasisFunc2D} but they have to fulfil additionally the Laplace equation in the linear parameter space of each face. Depending on the shapes of the elements and faces, the Laplace equations might be understood in a weak sense. Since $v\in H^1(K)$ locally and due to the continuity of $v$ across edges and faces for $v\in V_h$, the conformity $V_h\subset H^1(\Omega)$ follows. A further adaptation of the approximation space in three-dimensions can be found in~\cite{HofreitherLangerWeisser2016}. In order to achieve good approximation properties in $V_h$ the polytopal mesh and the elements in particular have to fulfil certain regularity assumptions.

\subsection{Characterisation of anisotropy and affine mapping}
\label{subsec:Characterisation}
Let $K\subset\mathbb R^d$, $d=2,3$ be a bounded polytopal element. Furthermore, we assume that $K$ is not degenerated, i.e.\ $|K|=\mathrm{meas}_d(K)>0$. Then, we define the center or mean of $K$ as
\[
  \bar{\vec x}_K = \frac{1}{|K|}\int_K \vec x\,d\vec x
\]
and the covariance matrix of $K$ as
\[
  \MCov(K) = \frac{1}{|K|}\int_K (\vec x-\bar{\vec x}_K)(\vec x-\bar{\vec x}_K)^\top\,d\vec x \in\mathbb R^{d\times d}.
\]
Obviously, $\MCov$ is real valued, symmetric and positive definite since $K$ is not degenerated. Therefore, it admits an eigenvalue decomposition
\[
  \MCov(K) = U_K\Lambda_KU_K^\top
\]
with
\[
  U^\top=U^{-1}
  \quad\mbox{and}\quad
  \Lambda_K=\diag(\lambda_{K,1},\ldots,\lambda_{K,d}).
\] 
Without loss of generality we assume that the eigenvalues fulfil $\lambda_{K,1}\geq\ldots\geq\lambda_{K,d}>0$ and that the eigenvectors $\vec u_{K,1},\ldots, \vec u_{K,d}$ collected in~$U$ are oriented in the same way for all considered elements $K\in\mathcal K_h$.

The eigenvectors of $\MCov(K)$ give the characteristic directions of $K$. This fact is, e.g., also used in the principal component analysis (PCA). The eigenvalue~$\lambda_{K,j}$ is the variance of the underlying data in the direction of the corresponding eigenvector~$\vec u_{K,j}$. Thus, the square root of the eigenvalues give the standard deviations in a statistical setting. Consequently, if 
\[
  \MCov(K) = cI
\]
for $c>0$, there are no dominant directions in the element $K$. We can characterise the anisotropy with the help of the quotient $\lambda_{K,1}/\lambda_{K,d}\geq 1$ and call an element 
\begin{align*}
  \text{isotropic, if}\quad & \frac{\lambda_{K,1}}{\lambda_{K,d}} \approx 1, \\
  \text{and anisotropic, if}\quad & \frac{\lambda_{K,1}}{\lambda_{K,d}} \gg 1.
\end{align*}
We might even characterise for $d=3$ whether the element is anisotropic in one or more directions by comparing the different combinations of eigenvalues.

Exploiting the spectral information of the polytopal elements, we next introduce a linear transformation of an anisotropic element $K$ onto a kind of reference element $\widehat{K}$. For each $\vec x\in K$, we define the mapping by
\begin{equation}\label{eq:trafo}
  \vec x \mapsto \widehat{\vec x} = F_K(\vec x) = A_K\vec x
  \quad\mbox{ with }\quad
  A_K=\alpha_K\Lambda_K^{-1/2}U_K^\top,
\end{equation}
and $\alpha>0$, which will be chosen later. $\widehat{K}=F_K(K)$ is called reference configuration later on.
\begin{lemma}\label{lem:PropTransform}
  Under the above transformation, it holds
  \begin{enumerate}
    \item $|\widehat K| = |K|\,|\det(A_K)| = \alpha^d|K|/\sqrt{\prod_{j=1}^d\lambda_{K,j}}$,
    \item $\bar{\vec x}_{\widehat{K}} = F_K(\bar{\vec x}_K)$,
    \item $\MCov(\widehat K) = \alpha^2I$.
  \end{enumerate}
\end{lemma}
\begin{proof}
  First, we recognize that
  \begin{eqnarray*}
    \det(A_K) = \alpha^d \det(\Lambda_K^{-1/2}U_K^\top)
    = \alpha^d /\sqrt{\det(\Lambda_K)}
    = \alpha^d /\sqrt{\textstyle\prod_{j=1}^d\lambda_{K,j}}.
  \end{eqnarray*}
  Consequently, we obtain by the transformation
  \begin{eqnarray*}
   |\widehat K| = \int_{\widehat K}d\widehat{\vec x}
    = |K|\, |\det(A_K)|
    = \alpha^d|K|/\sqrt{\det(\MCov(K))},
  \end{eqnarray*}
  that proves the first statement. For the center, we have
  \begin{equation*}
    \bar{\vec x}_{\widehat K} = \frac{1}{|\widehat K|}\int_{\widehat K}\widehat{\vec x}\,d\widehat{\vec x}
    = \frac{|\det(A_K)|}{|\widehat K|}\, A_K\int_{K}\vec x\,d\vec x
    = A_K\bar{\vec x}_K.
  \end{equation*}
  The covariance matrix has the form
  \begin{eqnarray*}
    \MCov(\widehat K) 
    &=& \frac{1}{|\widehat K|}\int_{\widehat K} (\widehat{\vec x}-\bar{\vec x}_{\widehat K})(\widehat{\vec x}-\bar{\vec x}_{\widehat K})^\top\,d\widehat{\vec x}\\
    &=& \frac{|\det(A_K)|}{|\widehat K|}\int_K A_K(\vec x-\bar{\vec x}_K)\,(A_K(\vec x-\bar{\vec x}_K))^\top\,d\vec x\\
    &=& A_K\MCov(K)A_K^\top\\
    &=& \alpha^2(\Lambda_K^{-1/2}U_K^\top)(U_K\Lambda_KU_K^\top)(\Lambda_K^{-1/2}U_K^\top)^\top\\
    &=& \alpha^2\,I,
  \end{eqnarray*}
  that finishes the proof.
\end{proof}

According to the previous lemma, the reference configuration $\widehat K$ is isotropic, since $\lambda_{\widehat K,1}/\lambda_{\widehat K,d}=1$, and thus, it has no dominant direction. We can still choose the parameter $\alpha$ in the mapping. We might use $\alpha=1$ such that the variance of the element in every direction is equal to one. On the other hand, we can use the parameter $\alpha$ in order to normalise the volume of $\widehat K$ such that $|\widehat K|=1$. This is achieved by
\begin{equation}\label{eq:Alpha}
  \alpha_K = \left(\frac{\sqrt{\det(\MCov(K))}}{|K|}\right)^{1/d}
  = \left(\frac{\sqrt{\prod_{j=1}^d\lambda_{K,j}}}{|K|}\right)^{1/d},
\end{equation}
see Lemma~\ref{lem:PropTransform}, and will be used in the rest of the paper.

\begin{example}
The transformation~\eqref{eq:trafo} for $\alpha$ according to~\eqref{eq:Alpha} is demonstrated for an anisotropic element $K\subset\mathbb R^2$, i.e. $d=2$. The element $K$ is depicted in Fig.~\ref{fig:Transform} (left). The eigenvalues of $\MCov(K)$ are
\[
  \lambda_{K,1} \approx 3.37\cdot10^{1}
  \quad\mbox{and}\quad
  \lambda_{K,2} \approx 1.69\cdot10^{-1},
\]
and thus
\[
  \frac{\lambda_{K,1}}{\lambda_{K,2}} \approx 1.99\cdot10^2 \gg 1.
\]
In Fig.~\ref{fig:Transform}, we additionally visualize the eigenvectors of $\MCov(K)$ scaled by the square root of their corresponding eigenvalue and centered at the mean of the element. The ellipse is the one given uniquely by the scaled vectors. In the right picture of Fig.~\ref{fig:Transform}, the transformed element $\widehat K=F_K(K)$ is given with the scaled eigenvectors of its covariant matrix $\MCov(\widehat K)$. The computation verify $|\widehat K|=1$, and we have
\[
\MCov(\widehat K) \approx
\begin{pmatrix}
	8.68\cdot10^{- 2} &  5.49\cdot10^{-16}\\
	5.49\cdot10^{-16} &  8.68\cdot10^{- 2}\\
\end{pmatrix}.
\]
\begin{figure}[tbp]
 {\small
 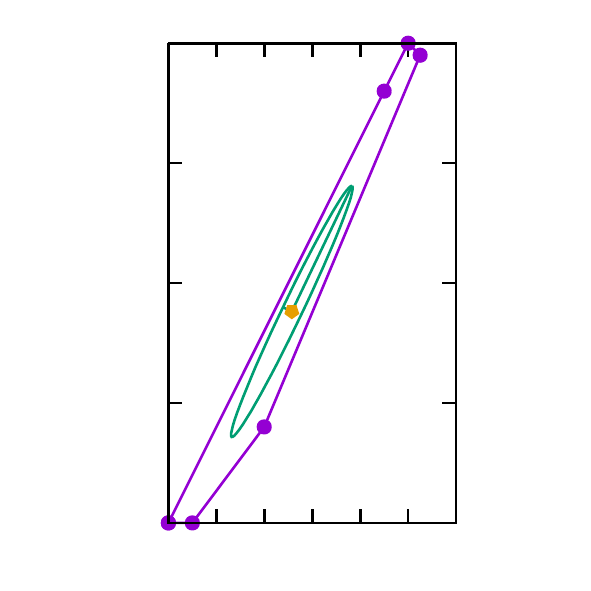\hfil
 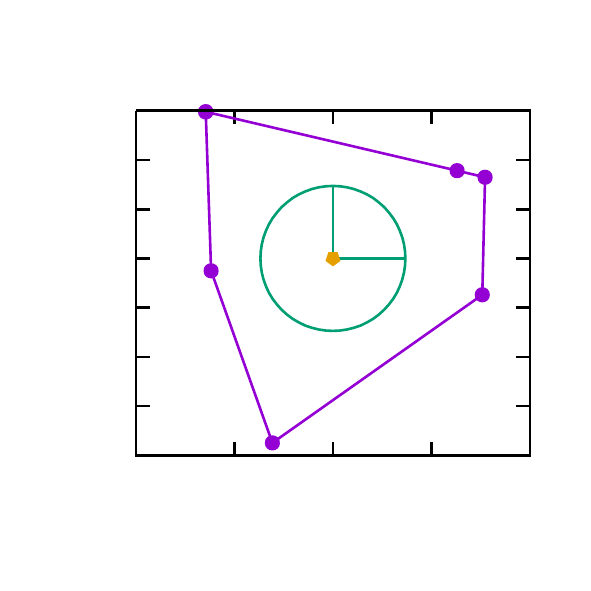
 }
 \caption{Demonstration of transformation~\eqref{eq:trafo}: original anisotropic element (left) and transformed element centered at the origin (right)}
 \label{fig:Transform}
\end{figure}
\end{example}

\subsection{Regular isotropic and anisotropic polytopal meshes}
\label{subsec:RegularMeshes}
In view of the quasi-interpolation and interpolation operators and their approximation properties, we state the mesh requirements for their analysis. For the regularity of usual isotropic meshes we refer to~\cite{RjasanowWeisser2014,Weisser2014}. However, these assumptions are rather standard, see also~\cite{BeiraoDaVeigaBrezziCangianiManziniMariniRusso2012}.

\begin{definition}[\bf regular (isotropic) mesh]\label{def:reg_isotropic_mesh}
Let $\mathcal K_h$ be a polytopal mesh. $\mathcal K_h$ is called \emph{regular} or a \emph{regular isotropic mesh}, if all elements $K\in\mathcal K_h$ fulfil:
\begin{enumerate}
\item\label{item:reg_iso_mesh:1} $K$ is a star-shaped polygon/polyhedron with respect to a circle/ball of radius~$\rho_K$ and midpoint~$\vec z_K$.
 \item\label{item:reg_iso_mesh:2} Their aspect ratio is uniformly bounded from above by $\sigma_\mathcal{K}$, i.e. $h_K/\rho_K<\sigma_\mathcal{K}$.
 \item\label{item:reg_iso_mesh:3} For the element $K$ and all its edges $e\subset\partial K$ it holds $h_K\leq c_\mathcal{K}|e|$, where $|e|$ is the edge length.
 \item\label{item:reg_iso_mesh:4} In the case $d=3$, all polygonal faces $E\subset\partial K$ of the polyhedral element~$K$ are star-shaped with respect to a circle of radius~$\rho_E$ and midpoint~$\vec z_E$ and their aspect ratio is uniformly bounded, i.e.\ $h_E/\rho_E<\sigma_{\mathcal E}$.
\end{enumerate}
\end{definition}
In~\cite{Weisser2014}, it has been shown that under these assumptions, the triangulation of a regular polygon~$K$, see Fig.~\ref{fig:auxTria}, obtained by connecting its nodes with the point~$\vec z_K$ is shape-regular in the sense of Ciarlet. The same holds for regular polyhedral elements. A discretization into tetrahedra is constructed by connecting the nodes of each face~$E\subset\partial K$ with $\vec z_E$, see Fig.~\ref{fig:auxTria}, and by connecting the vertices of the obtained triangles on $\partial K$ with the midpoint~$\vec z_K$. This tetrahedral decomposition of~$K$ is shape-regular in the usual sense, see~\cite{RjasanowWeisser2014}. Furthermore, it can be shown that the number of nodes on the boundary of~$K$ is uniformly bounded, cf.~\cite{Weisser2017}. Consequently, the number of simplices in the auxiliary triangulation into triangles ($d=2$) and tetrahedra ($d=3$) is also uniformly bounded.
\begin{figure}[tbp]
 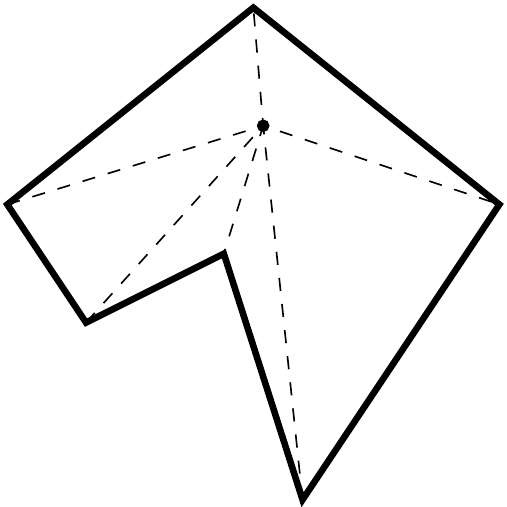\hfil
 \includegraphics[trim=4cm 2cm 4cm 4cm, width=0.35\textwidth, clip]{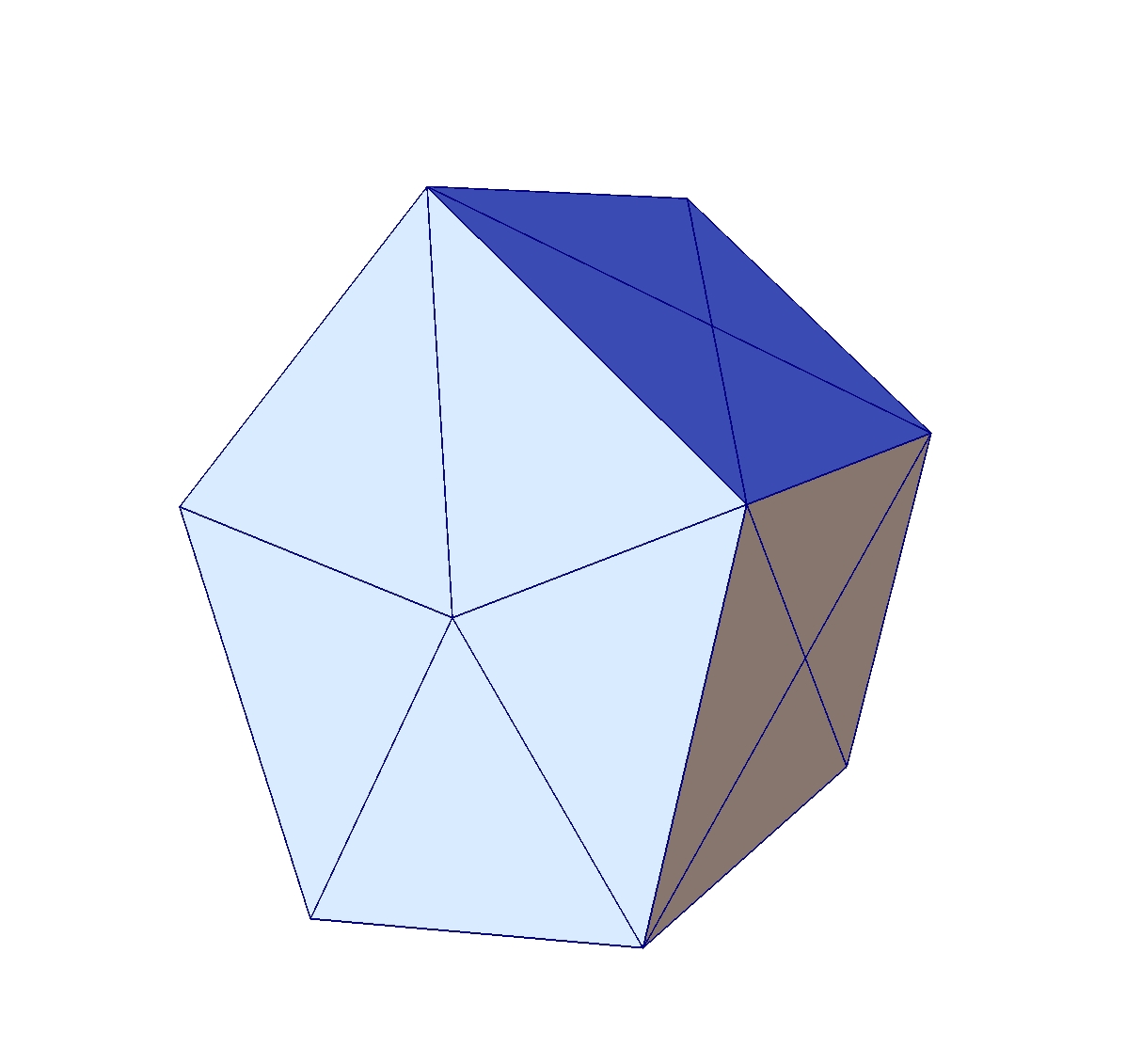}
 \caption{Auxiliary triangulation of regular element~$K$ for $d=2$ into triangles (left) and for $d=3$ into tetrahedra (right)}
 \label{fig:auxTria}
\end{figure}

In the definition of regular anisotropic meshes, we make use of the previously introduced reference configuration.
\begin{definition}[\bf regular anisotropic mesh]\label{def:reg_anisotropic_mesh}
 Let $\mathcal K_h$ be a polytopal mesh with anisotropic elements. $\mathcal K_h$ is called a \emph{regular anisotropic mesh}, if
 \begin{enumerate}
  \item The reference configuration $\widehat K$ for all $K\in\mathcal K_h$ obtained by~\eqref{eq:trafo} is a regular polytopal element according to Definition~\ref{def:reg_isotropic_mesh}.
  \item Neighbouring elements behave similarly in their anisotropy. More precisely, for two neighbouring elements $K_1$ and $K_2$, i.e. $\overline K_1\cap\overline K_2\neq\varnothing$, with covariance matrices 
   \[
     M_\mathrm{Cov}(K_1) = U_{K_1}\Lambda_{K_1}U_{K_1}^\top
     \quad\mbox{and}\quad
     M_\mathrm{Cov}(K_2) = U_{K_2}\Lambda_{K_2}U_{K_2}^\top
   \] 
   as defined above, we can write
   \[
     \Lambda_{K_2} = (I+\Delta^{K_1,K_2})\Lambda_{K_1}
     \quad\mbox{and}\quad
     U_{K_2} = R^{K_1,K_2}U_{K_1}
   \]
   with
   \[
     \Delta^{K_1,K_2} = \diag\left(\delta^{K_1,K_2}_j:j=1,\ldots,d\right),
   \]
   and a rotation matrix $R^{K_1,K_2}\in\mathbb R^{d\times d}$
   such that for $j=1,\ldots,d$
   \[
     0 \leq |\delta^{K_1,K_2}_j| < c_\delta < 1
     \quad\mbox{and}\quad
     0 \leq \|R^{K_1,K_2}-I\|_2\left(\frac{\lambda_{K_1,1}}{\lambda_{K_1,d}}\right)^{1/2} < c_R
   \]
   uniformly for all neighbouring elements, where $\|\cdot\|_2$ denotes the spectral norm.
 \end{enumerate}
\end{definition}
In the rest of the paper, $c$ denotes a generic constant which depends on the regularity parameters of the mesh ($\sigma_{\mathcal K}$, $c_{\mathcal K}$, $\sigma_{\mathcal E}$, $c_\delta$, $c_R$) and the space dimension~$d$.
\begin{remark}
For $d=2$, the rotation matrix has the form
\[
     R^{K_1,K_2} = \begin{pmatrix}
                     \cos\phi^{K_1,K_2} & -\sin\phi^{K_1,K_2}\\
                     \sin\phi^{K_1,K_2} &  \cos\phi^{K_1,K_2}
                   \end{pmatrix},
\]
with an angle~$\phi^{K_1,K_2}$. For the spectral norm $\|R^{K_1,K_2}-I\|_2$, we recognize that
\[
 (R^{K_1,K_2}-I)^\top (R^{K_1,K_2}-I)
 = \left(\sin^2\phi^{K_1,K_2} + (1-\cos\phi^{K_1,K_2})^2\right)I,
\]
and consequently
\begin{eqnarray*}
 \|R^{K_1,K_2}-I\|_2
 & = & \left(\sin^2\phi^{K_1,K_2} + (1-\cos\phi^{K_1,K_2})^2\right)^{1/2} \\
 & = & 2\left|\sin\left(\frac{\phi^{K_1,K_2}}{2}\right)-\sin(0)\right|\\
 & \leq & |\phi^{K_1,K_2}|,
\end{eqnarray*}
according to the mean value theorem. The assumption on the spectral norm in Definition~\ref{def:reg_anisotropic_mesh} can thus be replaced by
\[
  |\phi^{K_1,K_2}|\left(\frac{\lambda_{K_1,1}}{\lambda_{K_1,2}}\right)^{1/2} < c_\phi.
\]
This implies that neighbouring highly anisotropic elements has to be aligned in almost the same directions, whereas isotropic or moderately anisotropic elements might vary in their characteristic directions locally.
\end{remark}

Let us study the reference configuration~$\widehat K\subset\mathbb R^d$, $d=2,3$ of $K\in\mathcal K_h$, which is regular. Due to the scaling with $\alpha_K$, it is $|\widehat K|=1$ and we obtain
\[
  1 = |\widehat K| 
  \leq h_{\widehat K}^d
  \leq \sigma_{\mathcal K}^d\rho_{\widehat K}^d
  =    \sigma_{\mathcal K}^d\;\nu\pi\rho_{\widehat K}^d \;/\;(\nu\pi)
  \leq \sigma_{\mathcal K}^d\;|\widehat K| \;/\;(\nu\pi)
  =    \sigma_{\mathcal K}^d\;/\;(\nu\pi),
\]
where $\nu = 1$ for $d=2$ and $\nu=4/3$ for $d=3$, since the circle/ball is inscribed the element~$\widehat K$.
Consequently, we obtain
\begin{equation}\label{eq:BoundForRefDiameter}
  1 \leq h_{\widehat K} \leq \frac{\sigma_{\mathcal K}}{(\nu\pi)^{1/d}}.
\end{equation}
Furthermore, for $d=3$, let $\widehat E$ be a face of~$\widehat K$ and denote by $\widehat e$ one of its edges $\widehat e\subset\partial\widehat E$. Due to the regularity, we find
\[
  |\widehat E| 
  \geq \pi\rho_{\widehat E}^2
  \geq \pi h_{\widehat E}^2/\sigma_{\mathcal E}^2
  \geq \pi h_{\widehat e}^2/\sigma_{\mathcal E}^2
  \geq \pi h_{\widehat K}^2/(c_{\mathcal K}\sigma_{\mathcal E}^2),
\]
and thus for $d=2,3$
\begin{equation}\label{eq:BoundForRefDiameterDm1}
  h_{\widehat K}^{d-1} \leq c|\widehat E|.
\end{equation}

A regular anisotropic element can be mapped according to the previous definition onto a regular polytopal element in the usual sense. In the definition of quasi-interpolation operators, we deal, however, with patches of elements instead of single elements. Thus, we study the mapping of such patches. 
Let $\omega_i$ be the neighbourhood of the node $\vec x_i$ which is defined by
\[
  \overline\omega_i = \bigcup\left\{\overline{K'}: \vec x_i\in\overline{K'},\quad K'\in\mathcal K_h\right\}.
\]
The neighbourhood~$\omega_i$ is also described by
\[
  \overline\omega_i = \operatorname{supp}\psi_i,
\]
where $\psi_i$ is the nodal basis function in $V_h$ corresponding to~$\vec x_i$.
Furthermore, the neighbourhoods~$\omega_E$ and~$\omega_K$ of edges/faces~$E$ and elements~$K$ are considered. They are given by
\[
  \overline\omega_E = \bigcup_{i\in I(E)}\overline{\omega_i}
  \qquad\mbox{and}\qquad
  \overline\omega_K = \bigcup_{i\in I(K)}\overline{\omega_i}.
\]
\begin{lemma}\label{lem:regPerturbedMapping}
Let~$\mathcal K_h$ be a regular anisotropic mesh, $\omega_i$ be a patch as described above, and $K_1,K_2\in\mathcal K_h$ with $K_1,K_2\subset\omega_i$. The mapped element $F_{K_1}(K_2)$ is regular in the sense of Definition~\ref{def:reg_isotropic_mesh} with slightly perturbed regularity parameters $\widetilde\sigma_\mathcal K$ and $\widetilde c_\mathcal K$ depending only on the regularity of~$\mathcal K_h$. Consequently, the mapped patch $F_K(\omega_i)$ consists of regular polytopal elements for all $K\in\mathcal K_h$ with $K\subset\omega_i$. 
\end{lemma}
\begin{figure}[tbp]
 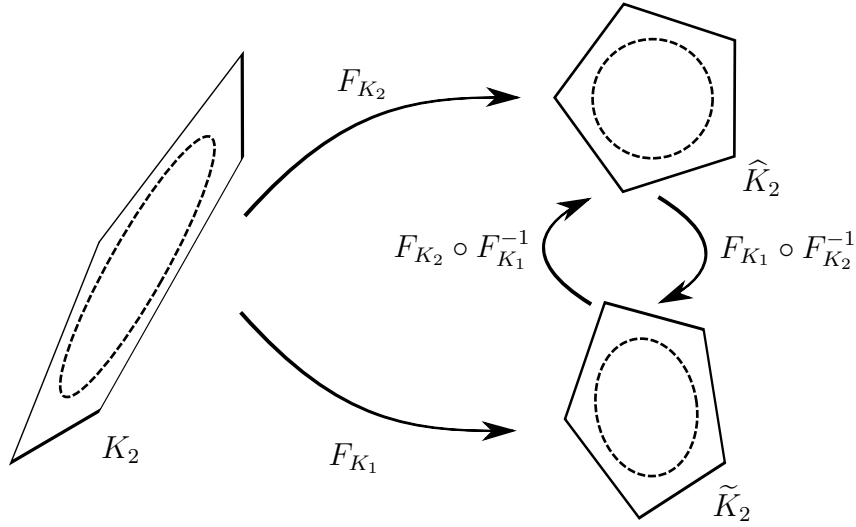
 \caption{Anisotropic element~$K_2$ with mapped regular element~$\widehat K_2$ and perturbed mapped element $\widetilde K_2=F_{K_1}(K_2)$}
 \label{fig:TransNeighbour}
\end{figure}
\begin{proof}
We verify Definition~\ref{def:reg_isotropic_mesh} for the mapped element $\widetilde K_2=F_{K_1}(K_2)$.

First, we address \ref{item:reg_iso_mesh:1} of Definition~\ref{def:reg_isotropic_mesh}.
$\widehat K_2=F_{K_2}(K_2)$ is regular and thus, star-shaped with respect to a circle/ball~$\widehat B$. If we transform $\widehat K_2$ into $\widetilde K_2$ with the mapping $F_{K_1}\circ F_{K_2}^{-1}$, see Fig.~\ref{fig:TransNeighbour}, the circle/ball $\widehat B$ is transformed into an ellipse/ellipsoid $\widetilde B=F_{K_1}\circ F_{K_2}^{-1}(\widehat B)$. Since the transformations are linear, the element $\widetilde K_2$ is star-shaped with respect to the ellipse/ellipsoid~$\widetilde B$ and in particular with respect to the circle/ball inscribed~$\widetilde B$.

Next, we address \ref{item:reg_iso_mesh:2} of Definition~\ref{def:reg_isotropic_mesh}
and we bound the aspect ratio. The radius~$\rho_{\widetilde K_2}$ of the inscribed circle/ball as above is equal to the smallest semi-axis of the ellipse/ellipsoid~$\widetilde B$. Let $\widetilde{\vec x}_1$ and $\widetilde{\vec x}_2$ be the intersection of $\widetilde B$ and the inscribed circle/ball. Thus, we obtain
\begin{eqnarray*}
 2\rho_{\widehat K_2} 
 & = & |F_{K_2}\circ F_{K_1}^{-1}(\widetilde{\vec x}_1-\widetilde{\vec x}_2)|\\
 & = & \left|\alpha_{K_2}\Lambda_{K_2}^{-1/2}U_{K_2}^\top\frac{1}{\alpha_{K_1}}U_{K_1}\Lambda_{K_1}^{1/2} (\widetilde{\vec x}_1-\widetilde{\vec x}_2)\right|\\
 & = & \frac{\alpha_{K_2}}{\alpha_{K_1}}\left|\Lambda_{K_1}^{-1/2}(I+\Delta^{K_1,K_2})^{-1/2}U_{K_1}^\top\left(R^{K_1,K_2}\right)^\top U_{K_1}\Lambda_{K_1}^{1/2} (\widetilde{\vec x}_1-\widetilde{\vec x}_2)\right|\\
 & = & \frac{\alpha_{K_2}}{\alpha_{K_1}}\left|(I+\Delta^{K_1,K_2})^{-1/2}\left(\Lambda_{K_1}^{-1/2}U_{K_1}^\top\left(R^{K_1,K_2}-I\right)^\top U_{K_1}\Lambda_{K_1}^{1/2}+I\right) (\widetilde{\vec x}_1-\widetilde{\vec x}_2)\right|\\
 & \leq & \frac{\alpha_{K_2}}{\alpha_{K_1}}\left\|(I+\Delta^{K_1,K_2})^{-1/2}\right\|_2\left(\|\Lambda_{K_1}^{-1/2}\|_2\|R^{K_1,K_2}-I\|_2 \|\Lambda_{K_1}^{1/2}\|_2+1\right) 2\rho_{\widetilde K_2}\\
 & = & \frac{\alpha_{K_2}}{\alpha_{K_1}}\max_{j=1,\ldots,d}\left\{(1+\delta_j^{K_1,K_2})^{-1/2}\right\}\left(1+\left(\frac{\lambda_{K_1,1}}{\lambda_{K_1,d}}\right)^{1/2}\|R^{K_1,K_2}-I\|_2\right) 2\rho_{\widetilde K_2},
\end{eqnarray*}
since the spectral norm~$\|\cdot\|_2$ is invariant under rotations, and $U_{K_1}$ is such a rotation.
With similar arguments, we can bound $h_{\widetilde K_2}$. Therefore, let $\widetilde{\vec x}_1, \widetilde{\vec x}_2\in\partial\widetilde K_2$ be such that $h_{\widetilde K_2}=|\widetilde{\vec x}_1-\widetilde{\vec x}_2|$ and $\widehat{\vec x}_i=F_{K_2}\circ F_{K_1}^{-1}(\widetilde{\vec x}_i)\in\partial\widehat K_2$, $i=1,2$. With similar considerations as above, we obtain
\begin{eqnarray*}
 h_{\widetilde K_2} 
 &   =  & |F_{K_1}\circ F_{K_2}^{-1}(\widehat{\vec x}_1-\widehat{\vec x}_2)|\\
 & \leq & \frac{\alpha_{K_1}}{\alpha_{K_2}}\max_{j=1,\ldots,d}\left\{(1+\delta_j^{K_1,K_2})^{1/2}\right\}\left(1+\left(\frac{\lambda_{K_1,1}}{\lambda_{K_1,d}}\right)^{1/2}\|R^{K_1,K_2}-I\|_2\right) h_{\widehat K_2}.
\end{eqnarray*}
Exploiting the last two estimates yields
\begin{eqnarray*}
 \frac{h_{\widetilde K_2}}{\rho_{\widetilde K_2}} 
 & \leq & \frac{\max_{j=1,\ldots,d}\sqrt{1+\delta_j^{K_1,K_2}}}{\min_{j=1,\ldots,d}\sqrt{1+\delta_j^{K_1,K_2}}} \left(1+\left(\frac{\lambda_{K_1,1}}{\lambda_{K_1,d}}\right)^{1/2}\|R^{K_1,K_2}-I\|_2\right)^2 \frac{h_{\widehat K_2}}{\rho_{\widehat K_2}} \\
 & \leq & \sqrt{\frac{1+c_\delta}{1-c_\delta}}\;(1+c_R)^2\; \frac{h_{\widehat K_2}}{\rho_{\widehat K_2}} 
   \leq   \sqrt{\frac{1+c_\delta}{1-c_\delta}}\;(1+c_R)^2\; \sigma_{\mathcal K}
   =      \widetilde\sigma_{\mathcal K}
\end{eqnarray*}
Obviously, the aspect ratio is uniformly bounded from above by a perturbed regularity parameter~$\widetilde\sigma_{\mathcal K}$.

Finally we address~\ref{item:reg_iso_mesh:3} of Definition~\ref{def:reg_isotropic_mesh}. 
Let $\widetilde e$ be an edge of $\widetilde K_2$ with endpoints $\widetilde{\vec x}_1$ and $\widetilde{\vec x}_2$. Furthermore, let $\widehat e$ be the corresponding edge of $\widehat K_2$ with endpoints $\widehat{\vec x}_1$ and $\widehat{\vec x}_2$. 
In the penultimate equation we estimated $h_{\widetilde K_2}$ by a term times $h_{\widehat K_2}$. Due to the regularity it is $h_{\widehat K_2}\leq c_{\mathcal K}|\widehat e|$ and, as in the estimate of $\rho_{\widehat K_2}$ above, we find that
\begin{eqnarray*}
 |\widehat e| 
 &   =  & |\widehat{\vec x}_1-\widehat{\vec x}_2|
 =    |F_{K_2}\circ F_{K_1}^{-1}(\widetilde{\vec x}_1-\widetilde{\vec x}_2)|\\
 & \leq & \frac{\alpha_{K_2}}{\alpha_{K_1}}\max_{j=1,\ldots,d}\left\{(1+\delta_j^{K_1,K_2})^{-1/2}\right\}\left(1+\left(\frac{\lambda_{K_1,1}}{\lambda_{K_1,d}}\right)^{1/2}\|R^{K_1,K_2}-I\|_2\right) |\widetilde e|.
\end{eqnarray*}
Summarizing, we obtain
\[
 h_{\widetilde K_2}
 \leq \sqrt{\frac{1+c_\delta}{1-c_\delta}}\;(1+c_R)^2\; c_{\mathcal K}|\widetilde e|
 =    \widetilde c_{\mathcal K}|\widetilde e|.
\]
\end{proof}

\begin{remark}
According to the previous proof, the perturbed regularity parameters are given by
\[
 \widetilde \sigma_{\mathcal K} = \sqrt{\frac{1+c_\delta}{1-c_\delta}}\;(1+c_R)^2\; \sigma_{\mathcal K}
 \quad\mbox{and}\quad
 \widetilde c_{\mathcal K} = \sqrt{\frac{1+c_\delta}{1-c_\delta}}\;(1+c_R)^2\; c_{\mathcal K}.
\]
\end{remark}

\begin{proposition}\label{prop:regPerturbedMappingPatches}
Let $K\in\mathcal K_h$ be a polytopal element of a regular anisotropic mesh~$\mathcal K_h$ and $E\in\mathcal E_h$ one of its edges ($d=2$) or faces ($d=3$). Then, the mapped patches $F_K(\omega_K)$ and $F_K(\omega_E)$ consist of regular polytopal elements.
\end{proposition}
\begin{proof}
The mapped patches $F_K(\omega_i)$, $i\in I(K)$ consist of regular polytopal elements according to Lemma~\ref{lem:regPerturbedMapping}. Since $\omega_K$ and $\omega_E$ are given as union of the neighbourhoods~$\omega_i$, the statement of the proposition follows.
\end{proof}

\begin{proposition}\label{prop:BoundedNumOfNodesVSElems}
Each node~$\vec x_i$ of a regular anisotropic mesh~$\mathcal K_h$ belongs to a uniformly bounded number of elements. Vice versa, each element $K\in\mathcal K_h$ has a uniformly bounded number of nodes on its boundary. 
\end{proposition}
\begin{proof}
Let~$\omega_i$ be the neighbourhood of the node~$x_i$. According to Lemma~\ref{lem:regPerturbedMapping}, the mapped neighbourhood~$\widetilde\omega_i$ consists of regular polytopal elements, which admit a shape-regular decomposition into simplices (triangles or tetrahedra). The mapped node~$\widetilde{\vec x}_i$ therefore belongs to a uniformly bounded number of simplices and thus to finitely many polytopal elements, cf.~\cite{Weisser2011,Weisser2017}. Since $\widetilde\omega_i$ is obtained by a linear transformation, we follow that $\vec x_i$ belongs to a uniformly bounded number of anisotropic elements. 
With the same argument we see that $\widetilde K$ and thus $K$ has a uniformly bounded number of nodes on its boundary.
\end{proof}

\begin{remark}
 In the publication of Apel and Kunert (see e.g.~\cite{Apel1999habil,Kunert2000}), it is assumed that neighbouring triangles/tetrahedra behave similarly. More precisely, they assume:
 \begin{itemize}
  \item The number of tetrahedra containing a node $\vec x_i$ is bounded uniformly.
  \item The dimension of adjacent tetrahedra must not change rapidly, i.e.
   \[h_{i,T}\sim h_{i,T'}\quad\forall T,T' \mbox{ with } T\cap T'\neq\varnothing,\;i=1,2,3,\]
   where $h_{1,T}\geq h_{2,T}\geq h_{3,T}$ are the heights of the tetrahedron~$T$ over its faces.
 \end{itemize}
 The first point is always fulfilled in our setting according to the previous proposition.
 The second point corresponds to our assumption that $\Lambda_{K_1}$ and $\Lambda_{K_2}$ differ moderately for neighbouring elements $K_1$ and $K_2$, see Definition~\ref{def:reg_anisotropic_mesh}. The assumption on $U_{K_1}$ and $U_{K_2}$ in the definition ensure that the heights are aligned in the same directions, this is 
 also hidden in the assumption of Apel and Kunert.
%
\end{remark}

The regularity of the mapped patches has several consequences, which are exploited in later proofs. 
\begin{lemma}\label{lem:UniformBounds}
Let $K_1,K_2\in\mathcal K_h$ be polytopal elements of a regular anisotropic mesh~$\mathcal K_h$, $\omega_i$ and $\omega_{K_1}$ be the neighbourhoods of the node~$\vec x_i$ and the element~$K_1$, respectively. 
Then, we have for the mapped patch $\widetilde\omega\in\{F_{K_1}(\omega_i),F_{K_1}(\omega_{K_1})\}$ and the neighbouring elements $K_1,K_2\subset\omega_i$, that
\[
	h_{\widetilde\omega} \leq c
	\qquad\mbox{and}\qquad
	\frac{|K_2|}{|K_1|} \leq c,
\]
where the constants only depend on the regularity parameters of the mesh.
\end{lemma}
\begin{proof}
According to Lemma~\ref{lem:regPerturbedMapping} and Proposition~\ref{prop:regPerturbedMappingPatches} the patch $\widetilde\omega$ consists of regular polytopal elements. 
Obviously, it is $h_{\widetilde\omega}\leq 2\max\{h_{\widetilde{K}}:\widetilde{K}\subset \widetilde\omega\}$. Let us assume without loss of generality that the maximum is reached for $\widetilde{K}$ which shares a common edge $\widetilde e$ with $\widetilde K_1$. Otherwise consider a sequence of polytopal elements in~$\widetilde\omega$, see~\cite{Weisser2011}. Due to the regularity of the elements, it is
\[
  h_{\widetilde\omega}
  \leq 2h_{\widetilde{K}}
  \leq 2c_\mathcal{K}|\widetilde{e}|
  \leq 2c_\mathcal{K}h_{\widetilde{K}_1}
  \leq \frac{2c_\mathcal{K}\sigma_{\mathcal K}}{(\nu\pi)^{1/d}}
\]
according to~\eqref{eq:BoundForRefDiameter}, since $\widetilde K_1=F_{K_1}(K_1)=\widehat K_1$.

In order to prove the second estimate, we observe that
$|K_1|=|\widehat K_1|/|\det(A_{K_1})|$, see Lemma~\ref{lem:PropTransform}. The same variable transform yields $|K_2|=|\widetilde K_2|/|\det(A_{K_1})|$, where $\widetilde K_2=F_{K_1}(K_2)$. Thus, we obtain
\[
  \frac{|K_2|}{|K_1|} 
  = \frac{|\widetilde K_2|}{|\widehat K_1|} 
  = |\widetilde K_2|
  \leq |\widetilde\omega_i|
  \leq h_{\widetilde\omega_i}^d
  \leq c.
\]
\end{proof}

\section{Anisotropic trace inequality and best approximation}
\label{sec:TraceInequalityAndBestApprox}
In this section we introduce some tools which are needed in later proofs. Here, the mapping~\eqref{eq:trafo} is employed to transform a regular anisotropic element~$K$ onto its reference configuration~$\widehat K$, which is regular in the sense of Definition~\ref{def:reg_isotropic_mesh}.

\begin{lemma}[\bf anisotropic trace inequality]\label{lem:AnisotropicTraceInequality}
 Let $K\in\mathcal K_h$ be a polytopal element of a regular anisotropic mesh~$\mathcal K_h$ with edge ($d=2$) or face ($d=3$) $E\in\mathcal E_h$, $E\subset\partial K$. It holds
 \[
  \|v\|_{L_2(E)}^2 \leq c\;\frac{|E|}{|K|}\left(\|v\|_{L_2(K)}^2+\|A_K^{-\top}\nabla v\|_{L_2(K)}^2\right),
 \]
 where the constant~$c$ only depends on the regularity parameters of the mesh.
\end{lemma}
\begin{proof}
In order to prove the estimate, we make use of the transformation~\eqref{eq:trafo} to the reference configuration~$\widehat K$ with $\widehat v=v\circ F_K^{-1}$, a trace inequality on~$\widehat K$, see~\cite{BeiraoDaVeigaLipnikovManzini2014,Weisser2011,Weisser2017}, as well as of~\eqref{eq:BoundForRefDiameter}, \eqref{eq:BoundForRefDiameterDm1} and $h_{\widehat K}^{-d}\leq |\widehat K|^{-1} = 1$:
\begin{eqnarray*}
  \|v\|_{L_2(E)}^2 
  & = &    \frac{|E|}{|\widehat E|}\|\widehat v\|_{L_2(\widehat E)}^2\\
  & \leq &  c\frac{|E|}{|\widehat E|}\left(h_{\widehat K}^{-1}\|\widehat v\|_{L_2(\widehat K)}^2+h_{\widehat K}|\widehat v|_{H^1(\widehat K)}^2\right) \\
  & \leq &  c|E|h_{\widehat K}^{-d}\left(\|\widehat v\|_{L_2(\widehat K)}^2+h_{\widehat K}^{2}|\widehat v|_{H^1(\widehat K)}^2\right) \\
  & \leq &  c|E|\left(\|\widehat v\|_{L_2(\widehat K)}^2+\|\widehat\nabla\widehat v\|_{L_2(\widehat K)}^2\right) \\
  & =    &  c\frac{|E|}{|K|}\left(\|v\|_{L_2(K)}^2+\|A^{-\top}\nabla v\|_{L_2(K)}^2\right).
\end{eqnarray*}
\end{proof}
\begin{remark}
 If we plug in the definition of $A=\alpha\Lambda_K^{-1/2}U_K^\top$, we have the anisotropic trace inequality
 \[
  \|v\|_{L_2(E)}^2 \leq c\frac{|E|}{|K|}\left(\|v\|_{L_2(K)}^2+\|\alpha^{-1}\Lambda_K^{1/2}U_K^\top\nabla v\|_{L_2(K)}^2\right).
 \]
 Obviously, the derivatives of $v$ in the directions $\vec u_{K,j}$ are scaled by $\lambda_j^{1/2}$, $j=1,\ldots,d$. This seems to be appropriate for functions with anisotropic behaviour which are aligned with the mesh.
\end{remark}

For later comparisons with other methods, we bound the term $|E|/|K|$ in case of $E\subset\partial K$. 
Let $\vec z_{\widehat K}$ be the midpoint of the circle/ball in Definition~\ref{def:reg_isotropic_mesh} of the regular reference configuration~$\widehat K$.
Obviously, it is $|K|\geq |P|$ for the $d$-dimensional pyramid with base side~$E$ and apex point $F_K^{-1}(z_{\widehat K})$, since $P\subset K$ due to the linearity of~$F_K$. Let $h_{P,E}$ be the hight of this pyramid, then it is $|P|=\frac13|E|h_{P,E}$ and we obtain
\begin{equation}\label{eq:BoundEbyK1}
 \frac{|E|}{|K|} \leq ch_{P,E}^{-1}.
\end{equation}

In the derivation of approximation estimates, the Poincar\'e constant plays a crucial role. For a domain $\omega\subset\mathbb R^d$, $d=2,3$, it is defined by
\begin{equation}\label{eq:PoincareConstant}
 C_P(\omega) = \sup_{v\in H^1(\omega)} \frac{\Vert v-\Pi_\omega v\Vert_{L_2(\omega)}}{h_{\omega} \vert v\vert_{H^1(\omega)}} < \infty,
\end{equation}
where $h_\omega$ is the diameter of $\omega$ and $\Pi_\omega$ denotes the $L_2$-projection over~$\omega$ into constants, i.e. 
\[
  \Pi_\omega v = \frac{1}{|\omega|}\int_\omega v(\vec x)d\vec x.
\]
\begin{lemma}\label{lem:AnisotropicPoincareOfMappedPatch}
Let~$\mathcal K_h$ be a regular anisotropic mesh, $\omega_i$ and $\omega_K$ be patches as described above, and $K\in\mathcal K_h$ with $K\subset\omega_i$. The Poincar\'e constants $C_P(\widetilde\omega_i)$ and $C_P(\widetilde\omega_K)$ for the mapped patches $\widetilde\omega_i=F_K(\omega_i)$ and $\widetilde\omega_K=F_K(\omega_K)$, respectively, can be bounded uniformly depending only on the regularity parameters of the mesh.
\end{lemma}
\begin{proof}
According to Lemma~\ref{lem:regPerturbedMapping} and Proposition~\ref{prop:regPerturbedMappingPatches}, the patches~$\widetilde\omega_i$ and~$\widetilde\omega_K$ consist of regular polytopal elements, which admit a shape-regular auxiliary triangulation with a uniformly bounded number of simplices. Thus, we can proceed as in~\cite{BeiraoDaVeigaLipnikovManzini2014} in order to prove the existence of a constant~$C^{Int}$ such that for $v\in H^1(\widetilde\omega)$, $\widetilde\omega\in\{\widetilde\omega_i,\widetilde\omega_K\}$ there exists a constant~$p$ such that
\[
 \|v-p\|_{L_2(\widetilde\omega)}+h_{\widetilde\omega}|v|_{H^1(\widetilde\omega)} \leq C^{Int}h_{\widetilde\omega}|v|_{H^1(\widetilde\omega)},
\]
where $C^{Int}$ only depends on the regularity of the triangulation and the number of simplices therein. Since $\|v-\Pi_{\widetilde\omega}v\|_{L_2(\widetilde\omega)}\leq\|v-p\|_{L_2(\widetilde\omega)}$, the statement of the proposition follows. See also the results from~\cite{Weisser2011,Weisser2017} for $d=2$.
\end{proof}

Next, we derive a best approximation result on patches of anisotropic elements.
\begin{lemma}\label{lem:AnisotropicPoincare}
Let~$\mathcal K_h$ be a regular anisotropic mesh with node~$\vec x_i$ and element $K\in\mathcal K_h$. Furthermore, let $\omega_i$ and $\omega_K$ be the neighbourhood of~$x_i$ and~$K$, respectively, and we assume $K\subset\omega_i$. For $\omega\in\{\omega_i,\omega_K\}$ it holds
\begin{equation*}
\|v-\Pi_{\omega}v\|_{L_2(\omega)} \leq c\;\|A_K^{-\top}\nabla v\|_{L_2(\omega)},
\end{equation*}
and furthermore
\begin{equation*}
\|v-\Pi_{\omega}v\|_{L_2(\omega)} \leq c\left(\sum_{K^\prime\in\mathcal K_h:K^\prime\subset\omega}\|A_{K^\prime}^{-\top}\nabla v\|_{L_2(K^\prime)}^2\right)^{1/2},
\end{equation*}
where the constant~$c$ only depends on the regularity parameters of the mesh.
\end{lemma}
\begin{proof}
We make use of the mapping~\eqref{eq:trafo} and indicate the objects on the mapped geometry with a tilde, e.g., $\widetilde\omega=F_K(\omega)$. Furthermore, we exploited that the mapped $L_2$-projection coincides with the $L_2$-projection on the mapped patch, i.e. $\widetilde{\Pi_{\omega}v}=\Pi_{\widetilde\omega}\widetilde v$. This yields together with Lemma~\ref{lem:AnisotropicPoincareOfMappedPatch} 
\begin{eqnarray*}
\|v-\Pi_{\omega}v\|_{L_2(\omega)}
& = &  |K|^{1/2}\;\|\widetilde v-\Pi_{\widetilde\omega}\widetilde v\|_{L_2(\widetilde\omega)} \\
& \leq &  ch_{\widetilde\omega}|K|^{1/2}\;|\widetilde v|_{H^1(\widetilde\omega)} \\
& = &  ch_{\widetilde\omega}|K|^{1/2}\;\|\widetilde\nabla\widetilde v\|_{L_2(\widetilde\omega)} \\
& = &  ch_{\widetilde\omega}\;\|A_K^{-\top}\nabla v\|_{L_2(\omega)}.
\end{eqnarray*}
The term $h_{\widetilde\omega}$ is uniformly bounded according to Lemma~\ref{lem:UniformBounds}, and thus the first estimate is proved.

In order to prove the second estimate, we employ the first one and write
\[
 \|v-\Pi_{\omega}v\|_{L_2(\omega)}^2
 \leq c\;\|A_K^{-\top}\nabla v\|_{L_2(\omega)}^2
 = c\sum_{K^\prime\in\mathcal K_h:K^\prime\subset\omega}\|A_{K}^{-\top}\nabla v\|_{L_2(K^\prime)}^2.
\]
Therefore, it remains to estimate $\|A_{K}^{-\top}\nabla v\|_{L_2(K^\prime)}$ by $\|A_{K^\prime}^{-\top}\nabla v\|_{L_2(K^\prime)}$ for any $K^\prime\subset\omega$. We make use of the mesh regularity, see Definition~\ref{def:reg_anisotropic_mesh}, and proceed similar as in the proof of Lemma~\ref{lem:regPerturbedMapping}.
\begin{eqnarray*}
 \lefteqn{\|A_{K}^{-\top}\nabla v\|_{L_2(K^\prime)}
     =    \frac{\alpha_{K^\prime}}{\alpha_{K}}\|\alpha_{K^\prime}^{-1}((I+\Delta^{K^\prime,K})\Lambda_{K^\prime})^{1/2}(R^{K^\prime,K}U_{K^\prime})^\top\nabla v\|_{L_2(K^\prime)}}\\
 &   =  & \frac{\alpha_{K^\prime}}{\alpha_{K}}\|\alpha_{K^\prime}^{-1}(I+\Delta^{K^\prime,K})^{1/2}\Lambda_{K^\prime}^{1/2}U_{K^\prime}^\top(R^{K^\prime,K})^\top\nabla v\|_{L_2(K^\prime)}\\
 &   =  & \frac{\alpha_{K^\prime}}{\alpha_{K}}\|\alpha_{K^\prime}^{-1}(I+\Delta^{K^\prime,K})^{1/2}\Lambda_{K^\prime}^{1/2}U_{K^\prime}^\top(R^{K^\prime,K})^\top U_{K^\prime}\Lambda_{K^\prime}^{-1/2}\Lambda_{K^\prime}^{1/2}U_{K^\prime}^\top\nabla v\|_{L_2(K^\prime)}\\
 & \leq & \frac{\alpha_{K^\prime}}{\alpha_{K}}\|(I+\Delta^{K^\prime,K})^{1/2}\Lambda_{K^\prime}^{1/2}U_{K^\prime}^\top(R^{K^\prime,K})^\top U_{K^\prime}\Lambda_{K^\prime}^{-1/2}\|_2\|A_{K^\prime}^{-\top}\nabla v\|_{L_2(K^\prime)},
\end{eqnarray*}
where we substituted $A_K^{-\top}=\alpha_{K^\prime}^{-1}\Lambda_{K^\prime}^{1/2}U_{K^\prime}^\top$.
Finally, we have to bound the ratio $\alpha_{K^\prime}/\alpha_{K}$ and the matrix norm.
According to the choice~\eqref{eq:Alpha} and Lemma~\ref{lem:UniformBounds}, it is
\begin{eqnarray*}
 \left(\frac{\alpha_{K^\prime}}{\alpha_{K}}\right)^2
 &   =  & \frac{|K|\sqrt{\prod_{j=1}^d\lambda_{K^\prime,j}}}{|K^\prime|\sqrt{\prod_{j=1}^d\lambda_{K,j}}}
     =    \frac{|K|\sqrt{\prod_{j=1}^d(1+\delta^{K,K^\prime}_j)\lambda_{K,j}}}{|K^\prime|\sqrt{\prod_{j=1}^d\lambda_{K,j}}}\\
 & \leq & (1+c_\delta)^{d/2}\frac{|K|}{|K^\prime|}
   \leq   c,
\end{eqnarray*}
and for the matrix norm, we have 
\begin{eqnarray*}
  \lefteqn{\|(I+\Delta^{K^\prime,K})^{1/2}\Lambda_{K^\prime}^{1/2}U_{K^\prime}^\top(R^{K^\prime,K})^\top U_{K^\prime}\Lambda_{K^\prime}^{-1/2}\|_2}\\
  & \leq & \|(I+\Delta^{K^\prime,K})^{1/2}\|_2 \|\Lambda_{K^\prime}^{1/2}U_{K^\prime}^\top(R^{K^\prime,K}-I)^\top U_{K^\prime}\Lambda_{K^\prime}^{-1/2}+I\|_2\\
  & \leq & \sqrt{1+c_\delta}(1+c_R),
\end{eqnarray*}
which finishes the proof.
\end{proof}
\begin{remark}\label{rem:NeighbouringElements}
 In the previous proof, we have seen in particular that for neighbouring elements $K,K^\prime\subset\omega_K$, it is 
 \[
  \|A_{K}^{-\top}\nabla v\|_{L_2(K^\prime)}
  \leq c\;\|A_{K^\prime}^{-\top}\nabla v\|_{L_2(K^\prime)}
 \]
 with a constant depending only on the regularity of the mesh.
\end{remark}

\section{Quasi-interpolation of non-smooth functions}
\label{sec:QuasiInterpolation}
In this section, we study quasi-interpolation operators on anisotropic polygonal and polyhedral meshes. Classical results on simplicial meshes with isotropic elements go back to Cl\'ement~\cite{Clement1975} and to Scott and Zhang~\cite{ScottZhang1990}. Quasi-interpolation operators on anisotropic simplicial meshes can be found in~\cite{Apel1999habil,Kunert2000}, for example. Cl\'ement-type interpolation operators on polygonal meshes have been studied in~\cite{Weisser2011,Weisser2017}.

Having the application to boundary value problems in mind, we split the boundary of the domain $\Omega\subset\mathbb R^d$, $d=2,3$ into a Dirichlet $\Gamma_D$ and Neumann $\Gamma_N$ part such that $\partial\Omega=\overline{\Gamma}_D\cup\overline{\Gamma}_N$. We consider the Sobolev space
\[
  H^1_D(\Omega) = \{v\in H^1(\Omega):\; v=0\mbox{ on } \Gamma_D\}
\]
of functions with vanishing trace on~$\Gamma_D$ and we allow $\Gamma_D=\partial\Omega$ and $\Gamma_D=\varnothing$.
Let $\mathcal K_h$ be a regular anisotropic mesh of~$\Omega$ such that the edges and faces~$\mathcal E_h$ are compatible with the Dirichlet and Neumann boundary. Furthermore, let the nodes~$\vec x_i$ be numbered such that they lie inside of~$\Omega$ for $i=1,\ldots,N_\mathrm{int}$, on~$\Gamma_N$ for $i=N_\mathrm{int}+1,\ldots,N_\mathrm{DoF}$ and on~$\overline\Gamma_D$ for $i=N_\mathrm{DoF}+1,\ldots,N$. If $\Gamma_N=\varnothing$, we set $N_\mathrm{DoF}=N_\mathrm{int}$.
The objective of this section is to define a quasi-interpolation operator
\[
  \mathfrak I:H_D^1(\Omega)\to V_h\cap H_D^1(\Omega),
\]
which fulfils anisotropic interpolation error estimates and which preserves the homogeneous Dirichlet data. The discrete space~$V_h$ is given as discussed in Sec.~\ref{subsec:DiscretizationSpace} with basis functions~$\psi_i$. Let $v\in H^1_D(\Omega)$, as usual we define
\begin{equation}\label{eq:QuasiInterpolation}
  \mathfrak Iv(\vec x) = \sum_{i=1}^{N_*}(\Pi_{\sigma_i}v)(\vec x_i)\;\psi_i(\vec x)
\end{equation}
for $\vec x\in\Omega$, where
$
  \Pi_{\sigma_i}: L_2(\sigma_i)\to\mathcal P^0(\sigma_i)
$
is the $L_2$-projection into the space of constants over~$\sigma_i$. The Cl\'ement and Scott-Zhang interpolation operators differ in the choice of~$\sigma_i$ and $N_*$.

\subsection{Cl\'ement-type interpolation}
\label{subsec:Clement}
The Cl\'ement interpolation operator~$\mathfrak I_C$ is defined as usual by~\eqref{eq:QuasiInterpolation}, where we choose $N_*=N_\mathrm{DoF}$ and $\sigma_i=\omega_i$. Thus, it is given as a linear combination of the basis functions~$\psi_i$ associated to the nodes in the interior of~$\Omega$ and the Neumann boundary~$\Gamma_N$. The expansion coefficients are chosen as average over the neighbourhood of the corresponding nodes. For $v\in H^1_D(\Omega)$, it is $\mathfrak I_Cv\in H^1_D(\Omega)$ by construction. 

Recall, that $I(K)$ and $I(E)$ denote the sets of indices of nodes which belong to the element~$K$ and the edge/face~$E$, respectively. Similarly, we denote by $I(\Gamma_D)$ the set of indices of the nodes which are located on the Dirichlet boundary~$\Gamma_D$. The following interpolation error estimates hold involving the neighbourhoods $\omega_K$ and $\omega_E$ of elements and edges/faces.

\begin{theorem}\label{th:ClementInterpolationEstimates}
Let~$\mathcal K_h$ be a regular anisotropic mesh with nodes~$\vec x_i$ as described above. Furthermore, let $\omega_i$ be the neighbourhood of~$\vec x_i$ and $K\in\mathcal K_h$. For $v\in H^1_D(\Omega)$, it is
\[
  \|v-\mathfrak I_Cv\|_{L_2(K)}\leq c\,\|A_K^{-\top}\nabla v\|_{L_2(\omega_K)},
\]
and for an edge/face $E\in\mathcal E_h$ with $E\subset\partial K\setminus\Gamma_D$
\[
  \|v-\mathfrak I_Cv\|_{L_2(E)}\leq c\,\,\frac{|E|^{1/2}}{|K|^{1/2}}\,\,\|A_{K}^{-\top}\nabla v\|_{L_2(\omega_E)},
\]
where the constants~$c$ only depend on the regularity parameters of the mesh.
\end{theorem}
\begin{proof}
We can follow classical arguments as for isotropic meshes, cf.~\cite{Verfuerth1996}, and see the adaptation to polygonal meshes in~\cite{Weisser2011}. The main ingredients are the observation that the basis functions~$\psi_i$ form a partition of unity, i.e.\ $\sum_{i\in I(K)}\psi_i=1$ on~$\overline K$, and that they fulfil $\|\psi_i\|_{L_\infty(\overline K)}=1$. Furthermore, anisotropic approximation estimates, see Lemma~\ref{lem:AnisotropicPoincare}, the trace inequality in Lemma~\ref{lem:AnisotropicTraceInequality}, Lemma~\ref{lem:UniformBounds} and Remark~\ref{rem:NeighbouringElements} are employed. We only sketch the proof of the second estimate.

The partition of unity property is used, which also holds on each edge/face~$E$, i.e.\ $\sum_{i\in I(E)}\psi_i=1$ on~$\overline E$.
We distinguish two cases, first let $I(E)\cap I(\Gamma_D)=\varnothing$. With the help of Lemma~\ref{lem:AnisotropicTraceInequality} and Lemma~\ref{lem:AnisotropicPoincare}, we obtain
\begin{eqnarray*}
 \|v-\mathfrak I_Cv\|_{L_2(E)} 
 &   =  & \sum_{i\in I(E)}\|\psi_i(v-\Pi_{\omega_i}v)\|_{L_2(E)}\,\,
   \leq   \sum_{i\in I(E)}\|v-\Pi_{\omega_i}v\|_{L_2(E)}\\
 & \leq & c\sum_{i\in I(E)}\frac{|E|^{1/2}}{|K|^{1/2}}\left(\|v-\Pi_{\omega_i}v\|_{L_2(K)}^2+\|A_{K}^{-\top}\nabla v\|_{L_2({K})}^2\right)^{1/2}\\
 & \leq & c\sum_{i\in I(E)}\frac{|E|^{1/2}}{|K|^{1/2}}\,\,\|A_{K}^{-\top}\nabla v\|_{L_2(\omega_i)}.
\end{eqnarray*}
For the second case $I(E)\cap I(\Gamma_D)\neq\varnothing$, we find 
\begin{equation}\label{eq:Est2ClementNearDirichleBdry}
 \|v-\mathfrak I_Cv\|_{L_2(E)}  \leq \sum_{i\in I(E)}\|\psi_i(v-\Pi_{\omega_i}v)\|_{L_2(E)} + \sum_{i\in I(E)\cap I(\Gamma_D)}\|\psi_i\Pi_{\omega_i}v\|_{L_2(E)}.
\end{equation}
The first sum has already been estimated, thus we consider the term in the second sum. For $i\in I(E)\cap I(\Gamma_D)$, i.e.\ $\vec x_i\in\overline\Gamma_D$, there is an element $K'\subset\omega_i$ and an edge/face $E'\subset\overline{K'}\cap\Gamma_D$ such that $\vec x_i\in\overline{E'}$. Since $v$ vanishes on $E'$, Lemmata~\ref{lem:AnisotropicTraceInequality} and~\ref{lem:AnisotropicPoincare} as well as Remark~\ref{rem:NeighbouringElements} yield
\[
 \vert \Pi_{\omega_i}v\vert 
 = |E'|^{-1/2}\,\|v-\Pi_{\omega_i}v\|_{L_2(E')}
 \leq c\,|K^\prime|^{-1/2}\,\|A_{K}^{-\top}\nabla v\|_{L_2(\omega_i)}.
\]
Because $|K^\prime|/|K|$ is uniformly bounded according to Lemma~\ref{lem:UniformBounds}, we obtain
\[
 \|\psi_i\Pi_{\omega_i}v\|_{L_2(E)} 
 \leq \vert \Pi_{\omega_i}v\vert\,\|\psi_i\|_{L_\infty(E)}\,\vert E\vert^{1/2} 
 \leq c\,\frac{|E|^{1/2}}{|K|^{1/2}}\,\|A_{K}^{-\top}\nabla v\|_{L_2(\omega_i)}.
\]
Finally, since the number of nodes per element is uniformly bounded according to Proposition~\ref{prop:BoundedNumOfNodesVSElems}, this estimate as well as the one derived in the first case applied to~\eqref{eq:Est2ClementNearDirichleBdry} yield the second interpolation error estimate in the theorem.
\end{proof}

\begin{remark}\label{rem:CaseIsotropic}
 In the case of an isotropic polytopal element~$K$ with edge/face $E$ it is
 \[
  \lambda_1\approx\ldots\approx\lambda_d\sim h_K^2, 
  \quad\mbox{and thus}\quad
  \alpha_K \sim 1.
 \]
 Therefore, we obtain from Theorem~\ref{th:ClementInterpolationEstimates} with $A_K^{-\top}=\alpha_K^{-1}\Lambda_K^{1/2}U_K^\top$ that
 \[
   \|v-\mathfrak I_Cv\|_{L_2(K)}
   \leq ch_K\,\|U_K^\top\nabla v\|_{L_2(\omega_K)}
    =   ch_K\,|v|_{H^1(\omega_K)},
 \]
 and
 \[
  \|v-\mathfrak I_Cv\|_{L_2(E)}
  \leq c\,\,\frac{h_K|E|^{1/2}}{|K|^{1/2}}\,\,\|U_{K}^\top\nabla v\|_{L_2(\omega_E)}
  \leq ch_E^{1/2}\,|v|_{H^1(\omega_E)},
 \]
 since $|E|\leq h_E^{d-1}$ as well as $|K|\geq ch_K^d$ and $h_K\leq ch_E$ in consequence of Definition~\ref{def:reg_isotropic_mesh}.
 Obviously, we recover the classical interpolation error estimates for the Cl\'ement interpolation operator, cf., e.g., \cite{Verfuerth1996,Weisser2017}.
\end{remark}

In the following, we rewrite our results in order to compare them with the work of Formaggia and Perotto~\cite{FormaggiaPerotto2003}. 
It is $A_K^{-\top}=\alpha_K^{-1}\Lambda_K^{1/2}U_K^\top$ with $U_K=(\vec u_{K,1},\ldots, \vec u_{K,d})$. Thus, we observe
\[
 \|A_K^{-\top}\nabla v\|_{L_2(\omega_K)}^2
 = \alpha_K^{-2} \sum_{j=1}^d
    \lambda_{K,j}\,\|\vec u_{K,j}\cdot\nabla v\|_{L_2(\omega_K)}^2,
\]
and since $\vec u_{K,j}\cdot\nabla v\in\mathbb R$, we obtain
\[
  \|\vec u_{K,j}\cdot\nabla v\|_{L_2(\omega_K)}^2
    =   \sum_{K^\prime\subset\omega_K}\int_{K^\prime} \vec u_{K,j}^\top\nabla v(\nabla v)^\top\vec u_{K,j}\,d\vec x
    =   \vec u_{K,j}^\top\, G_K(v)\, \vec u_{K,j}
\]
with
\[
  G_K(v)
  = \sum_{K^\prime\subset\omega_K}
    \left(
      \int_{K^\prime} \frac{\partial v}{\partial x_i}\frac{\partial v}{\partial x_j}\,d\vec x 
    \right)_{i,j=1}^d\in\mathbb R^{d\times d},
  \quad
  \vec x = (x_1,\ldots,x_d)^\top.
\]
Therefore, we can deduce from Theorem~\ref{th:ClementInterpolationEstimates} an equivalent formulation.
\begin{proposition}\label{prop:ClementInterpolationEstimates}
Let $K\in\mathcal K_h$ be a polytopal element of a regular anisotropic mesh. The Cl\'ement interpolation operator fulfils for $v\in H^1(\Omega)$ the interpolation error estimate
\[
  \|v-\mathfrak I_Cv\|_{L_2(K)}
   \leq c\, \alpha_K^{-1}\left(
      \sum_{j=1}^d \lambda_{K,j}\,\vec u_{K,j}^\top\, G_K(v)\, \vec u_{K,j}
     \right)^{1/2},
\]
and
\[
  \|v-\mathfrak I_Cv\|_{L_2(E)}
   \leq c\, \alpha_K^{-1}\,\,\frac{|E|^{1/2}}{|K|^{1/2}}\left(
      \sum_{j=1}^d\lambda_{K,j}\,\vec u_{K,j}^\top\, G_K(v)\, \vec u_{K,j}
     \right)^{1/2},
\]
where the constant~$c$ only depends on the regularity parameters of the mesh.
\end{proposition}

Now we are ready to compare the interpolation error estimates with the ones derived by Formaggia and Perotto. These authors considered the case of anisotropic triangular meshes, i.e.\ $d=2$. The inequalities in Proposition~\ref{prop:ClementInterpolationEstimates} correspond to the derived estimates~(2.12) and~(2.15) in~\cite{FormaggiaPerotto2003} but they are valid on much more general meshes. When comparing these estimates to the results of Formaggia and Perotto, one has to take care on the powers of the lambdas. The triangular elements in their works are scaled with $\lambda_{i,K}$, $i=1,2$ in the characteristic directions whereas the scaling in this paper is $\lambda_{K,i}^{1/2}$, $i=1,2$.

Obviously, the first inequality of the previous proposition corresponds to the derived estimate~(2.12) in~\cite{FormaggiaPerotto2003} up to the scaling factor $\alpha_K^{-1}$. However, for convex elements the assumption 
\[
 \alpha_K \sim 1,
 \qquad\mbox{i.e.,}\qquad
 |K| \sim \sqrt{\lambda_{K,1}\lambda_{K,2}},
\]
seems to be convenient, since this means that the area of the element~$|K|$ is proportional to the area of the inscribed ellipse $\pi\sqrt{\lambda_{K,1}}\sqrt{\lambda_{K,2}}$, which is given by the scaled characteristic directions of the element. 

In order to recognize the relation of the second inequality under these assumptions, we estimate the term $|E|/|K|$ by~\eqref{eq:BoundEbyK1} and by applying $h_{P,E}\geq\lambda_{K,2}^{1/2}$.
This yields
\[
  \|v-\mathfrak I_Cv\|_{L_2(E)}
   \leq c\left(\frac{1}{\lambda_{K,2}^{1/2}}\right)^{1/2}\left(
      \lambda_{K,1}\,\vec u_{K,1}^\top\, G_K(v)\, \vec u_{K,1}+
      \lambda_{K,2}\,\vec u_{K,2}^\top\, G_K(v)\, \vec u_{K,2}
     \right)^{1/2},
\]
and shows the correspondence to~\cite{FormaggiaPerotto2003}, since $h_K$ and $\lambda_{1,K}$ are proportional in the referred work.

\subsection{Scott-Zhang-type interpolation}
The Scott-Zhang interpolation operator $\mathfrak I_{SZ}:H^1(\Omega)\to V_h$ is defined as usual by~\eqref{eq:QuasiInterpolation}, where we choose $N_*=N$ and $\sigma_i=E$, where $E\in\mathcal E_h$ is an edge ($d=2$) or face ($d=3$) with $\vec x_i\in\overline E$ and
\[
   E\subset\Gamma_D \mbox{ if }  \vec x_i\in\overline\Gamma_D 
  \qquad\mbox{and}\qquad
   E\subset\Omega\cup\Gamma_N \mbox{ if }  \vec x_i\in\Omega\cup\Gamma_N .
\]
Thus, the interpolation is given as a linear combination of all basis functions~$\psi_i$. The expansion coefficients are chosen as average over edges and faces. For $v\in H^1_D(\Omega)$, it is $\mathfrak I_Cv\in H^1_D(\Omega)$ by construction such that homogeneous Dirichlet data is preserved. 
We have the following local stability result, which can be utilized to derive interpolation error estimates.
\begin{lemma}\label{lem:LocalStabilitySZ}
Let $K\in\mathcal K_h$ be a polytopal element of a regular anisotropic mesh. The Scott-Zhang interpolation operator fulfils for $v\in H^1(\Omega)$ the local stability
\[
 \|\mathfrak I_{SZ}v\|_{L_2(K)} 
 \leq c\left(\|v\|_{L_2(\omega_K)}+\|A_{K}^{-\top}\nabla v\|_{L_2(\omega_K)}\right),
\]
where the constant~$c$ only depends on the regularity parameters of the mesh.
\end{lemma}
\begin{proof}
Due to the stability of the $L_2$-projection $\Pi_{\sigma_i}$ 
we have $\|\Pi_{\sigma_i}v\|_{L_2(\sigma_i)}\leq\|v\|_{L_2(\sigma_i)}$. Furthermore, there exists $K_i\in\mathcal K_h$ with $\sigma_i\subset\partial K_i$ such that $K_i\subset\omega_K$. Therefore, we obtain with the anisotropic trace inequality, see Lemma~\ref{lem:AnisotropicTraceInequality},
\begin{eqnarray*}
 |\Pi_{\sigma_i}v| 
 & = &    |\sigma_i|^{-1/2}\;\|\Pi_{\sigma_i}v\|_{L_2(\sigma_i)}\\
 & \leq & c\;|K_i|^{-1/2}\left(\|v\|_{L_2(K_i)}^2+\|\alpha_{K_i}^{-1}\Lambda_{K_i}^{1/2}U_{K_i}^\top\nabla v\|_{L_2(K_i)}^2\right)^{1/2},
\end{eqnarray*}
since $\sigma_i=E\in\mathcal E_h$.
Utilizing this estimate and $\|\psi_i\|_{L_\infty(K)}=1$ yields
\begin{eqnarray*}
 \|\mathfrak I_{SZ}v\|_{L_2(K)} 
 & \leq & \sum_{i\in I(K)} \|(\Pi_{\sigma_i}v)(\vec x_i)\;\psi_i\|_{L_2(K)} \\
 & \leq & \sum_{i\in I(K)} |(\Pi_{\sigma_i}v)|\;\|\psi_i\|_{L_\infty(K)}\;|K|^{1/2} \\
 & \leq & c\sum_{i\in I(K)} \left(\|v\|_{L_2(K_i)}^2+\|A_{K_i}^{-\top}\nabla v\|_{L_2(K_i)}^2\right)^{1/2},
\end{eqnarray*}
where we have used $|K|/|K_i|\leq c$ according to Lemma~\ref{lem:UniformBounds}. Applying the Cauchy-Schwarz inequality, Remark~\ref{rem:NeighbouringElements} and exploiting that the number of nodes per element is uniformly bounded, see Proposition~\ref{prop:BoundedNumOfNodesVSElems}, finishes the proof.
\end{proof}

\begin{theorem}
Let $K\in\mathcal K_h$ be a polytopal element of a regular anisotropic mesh. The Scott-Zhang interpolation operator fulfils for $v\in H^1(\Omega)$ the interpolation error estimate
\[
  \|v-\mathfrak I_{SZ}v\|_{L_2(K)} \leq \|A_{K}^{-\top}\nabla v\|_{L_2(\omega_K)},
\]
where the constant~$c$ only depends on the regularity parameters of the mesh.
\end{theorem}
\begin{proof}
For $p=\Pi_{\omega_K}v\in\mathbb R$ it is obviously $p=\mathfrak I_{SZ}p$ and $\nabla p = 0$.
The estimate in the theorem follows by Lemma~\ref{lem:LocalStabilitySZ} and the application of Lemma~\ref{lem:AnisotropicPoincare}, since
\begin{eqnarray*}
 \|v-\mathfrak I_{SZ}v\|_{L_2(K)} 
 & \leq & \|v-p\|_{L_2(K)} + \|\mathfrak I_{SZ}(v-p)\|_{L_2(K)} \\
 & \leq & c\left(\|v-p\|_{L_2(\omega_K)}+
                   \|A_{K}^{-\top}\nabla v\|_{L_2(\omega_K)}\right)\\
 & \leq & c\,\|A_{K}^{-\top}\nabla v\|_{L_2(\omega_K)}.
\end{eqnarray*}
\end{proof}

\section{Pointwise interpolation of smooth functions}
\label{sec:Interpolation}
In the previous section, we considered quasi-interpolation of functions in $H^1(\Omega)$. However, we may also address classical interpolation employing point evaluations in the case that the function to be interpolated is sufficiently regular. This is possible for functions in $H^2(\Omega)$. In the following, we consider the pointwise interpolation of lowest order
\begin{equation}\label{eq:PWInterpolationOp}
 \mathfrak I_\mathrm{pw}v(\vec x) = \sum_{i=1}^N v(\vec x_i)\;\psi_i(\vec x)
\end{equation}
for $v\in H^2(\Omega)$ on anisotropic meshes. For the analysis it is sufficient to study the restriction of $\mathfrak I_\mathrm{pw}:H^2(\Omega)\to V_h$ onto a single element $K\in\mathcal K_h$ and we denote this restriction by the same symbol
\[
 \mathfrak I_\mathrm{pw}:H^2(K)\to V_h|_{K}.
\]
Furthermore, we make use of the mapping to and from the reference configuration, cf.\ \eqref{eq:trafo}.
As earlier, we mark the operators and functions defined over the reference configuration by a hat, as, for instance,
$\widehat v = v\circ F_K^{-1}:\widehat K \to K$.
We already used
$\nabla v = \alpha_KU_K\Lambda_K^{-1/2}\widehat\nabla \widehat v$,
and by employing some calculus we find
\begin{equation}\label{eq:IdentityHessian}
 \widehat H(\widehat v) = \alpha_K^{-2}\Lambda_K^{1/2}U_K^\top H(v)U_K\Lambda_K^{1/2},
\end{equation}
where $H(v)$ denotes the Hessian matrix of~$v\in H^2(\Omega)$ and $\widehat H(\widehat v)$ the corresponding Hessian on the reference configuration.
\begin{lemma}\label{lem:MapH1norm}
 Let $K\in\mathcal K_h$ be a polytopal element of a regular anisotropic mesh~$\mathcal K_h$. For $v\in H^1(K)$, it is
 \[
  \sqrt{\frac{\prod_{j=2}^d\lambda_{K,j}}{\lambda_{K,1}}}\;|\widehat v|_{H^1(\widehat K)}^2 
  \leq |v|_{H^1(K)}^2 
  \leq \sqrt{\frac{\prod_{j=1}^{d-1}\lambda_{K,j}}{\lambda_{K,d}}}\;|\widehat v|_{H^1(\widehat K)}^2.
 \]
\end{lemma}
\begin{proof}
Applying the transformation to the reference configuration yields
\begin{eqnarray*}
 |v|_{H^1(K)}^2 
 &   =  & \|\nabla v\|_{L_2(K)}^2 
     =    |K|\;\|\alpha_KU_K\Lambda_K^{-1/2}\widehat\nabla \widehat v\|_{L_2(\widehat K)}^2\\ 
 &   =  & |K|\alpha_K^2\;\|\Lambda_K^{-1/2}\widehat\nabla \widehat v\|_{L_2(\widehat K)}^2 
     =    |K|\alpha_K^2 \sum_{j=1}^d \lambda_{K,j}^{-1}\left\|\frac{\partial\widehat v}{\partial\widehat x_j}\right\|_{L_2(\widehat K)}^2.
\end{eqnarray*}
Since $\lambda_{K,1}\geq\ldots\geq\lambda_{K,d}$, we obtain
\[
 \frac{|K|\alpha_K^2}{\lambda_{K,1}}\;|\widehat v|_{H^1(\widehat K)}^2
 \leq |v|_{H^1(K)}^2 
 \leq \frac{|K|\alpha_K^2}{\lambda_{K,d}}\;|\widehat v|_{H^1(\widehat K)}^2.
\]
Due to the choice~\eqref{eq:Alpha} for $\alpha_K$, it is
$|K|\alpha_K^2 = \sqrt{\prod_{j=1}^d\lambda_{K,j}}$,
that completes the proof.
\end{proof}

In order to derive interpolation error estimates, we make use of interpolation results on isotropic polytopal elements which are regular in the sense of Definition~\ref{def:reg_isotropic_mesh}, see, e.g., \cite{Weisser2014} or related works on VEM and generalized barycentric coordinates. First, we recognize the relation between the interpolation $\mathfrak I_\mathrm{pw}v$ transferred to the reference configuration~$\widehat K$ and the interpolation $\widehat{\mathfrak I}_\mathrm{pw}\widehat v$ defined directly on~$\widehat K$. Namely, it is
\begin{equation}\label{eq:PwIntOnKandHatK}
\widehat{\mathfrak I_\mathrm{pw}v} = \widehat{\mathfrak I}_\mathrm{pw}\widehat v,
\end{equation}
since only function evaluations in the nodes are involved. For the operator~$\widehat{\mathfrak I}_\mathrm{pw}$, we can apply known results. We use the convention that $H^0(K)=L_2(K)$.

\begin{theorem}\label{th:PointwiseInterpolationEstimate}
 Let $K\in\mathcal K_h$ be a polytopal element of a regular anisotropic mesh~$\mathcal K_h$. For $v\in H^2(\Omega)$, it is
 \[
  |v-\mathfrak I_\mathrm{pw}v|_{H^\ell(K)}^2
  \leq c \alpha_K^{-4}\,S_\ell(K)\sum_{i,j=1}^d\lambda_{K,i}\lambda_{K,j}L_K(\vec u_{K,i},\vec u_{K,j};v)
 \]
 with
 \[
  S_\ell(K) = \begin{cases}
                1, & \mbox{for } \ell=0,\\
                \displaystyle\frac{1}{|K|}\sqrt{\frac{\prod_{j=1}^{d-1}\lambda_{K,j}}{\lambda_{K,d}}}, & \mbox{for } \ell=1,\\
              \end{cases}
 \]
 where
 \[
  L_K(\vec u_{K,i},\vec u_{K,j};v) 
  = \int_K \left(\vec u_{K,i}^\top H(v)\vec u_{K,j}\right)^2d\vec x
  \quad\mbox{for } i,j=1,\ldots,d.
 \]
\end{theorem}
\begin{proof}
Property~\eqref{eq:PwIntOnKandHatK} together with the scaling to the reference configuration and Lemma~\ref{lem:MapH1norm} as well as~\eqref{eq:BoundForRefDiameter} yield
\begin{eqnarray*}
 |v-\mathfrak I_\mathrm{pw}v|_{H^0(K)}^2 
 & \leq & |K|\;S_\ell(K)\;\|\widehat v-\widehat{\mathfrak I}_\mathrm{pw}\widehat v\|_{L_2(\widehat K)}^2\\
 & \leq & ch_{\widehat K}^{2(2-\ell)}|K|\;S_\ell(K)\;|\widehat v|_{H^2(\widehat K)}^2\\
 & \leq & c|K|\;S_\ell(K)\;|\widehat v|_{H^2(\widehat K)}^2,
\end{eqnarray*}
where known interpolation estimates have been applied on the regular isotropic element~$\widehat K$, see, e.g.,~\cite{RjasanowWeisser2014,Weisser2014}. Next, we transform the $H^2$-semi-norm back to the element~$K$. Employing the mapping and the relation~\eqref{eq:IdentityHessian} gives
\begin{eqnarray*}
 |\widehat v|_{H^2(\widehat K)}^2
 & = & \int_{\widehat K}\|\widehat H(\widehat v)\|_F^2d\widehat{\vec x}\\
 & = & \frac{\alpha_K^{-4}}{|K|}\int_K\|\Lambda_K^{1/2}U_K^\top H(v)U_K\Lambda_K^{1/2}\|_F^2\,d\vec x,
\end{eqnarray*}
where $\|\cdot\|_F$ denotes the Frobenius norm of a matrix. A small exercise yields
\[
 \|\Lambda_K^{1/2}U_K^\top H(v)U_K\Lambda_K^{1/2}\|_F^2
 = \sum_{i,j=1}^d \lambda_{K,i}\lambda_{K,j}\left(\vec u_{K,i}^\top H(v)\vec u_{K,j}\right)^2,
\]
and consequently
\[
 |\widehat v|_{H^2(\widehat K)}^2
 = \frac{\alpha_K^{-4}}{|K|} \sum_{i,j}^d \lambda_{K,i}\lambda_{K,j}L_K(\vec u_{K,i},\vec u_{K,j};v).
\]
Combining the derived results yields the desired estimates.
\end{proof}

For the comparison with the work of Formaggia and Perotto we remember that their lambdas behave like $\lambda_{i,K}\sim\sqrt{\lambda_{K,i}}$, $i=1,2$. Employing the assumption $\alpha_K \sim 1$ raised in the comparison of Section~\ref{subsec:Clement}, we find
\[
 \frac{\sqrt{\lambda_{K,1}/\lambda_{K,2}}}{|K|}
 \sim \frac{1}{\lambda_{K,2}}.
\]
Therefore, we recognize that the estimates in Theorem~\ref{th:PointwiseInterpolationEstimate} match the results of Lemma~2 in~\cite{FormaggiaPerotto2003}, but on much more general meshes.

\section{Numerical assessment: anisotropic polytopal meshes}
\label{sec:AdaptiveRefinement}
In the introduction we already mentioned that polygonal and polyhedral meshes are much more flexible in meshing than classical finite element shapes. This is in particular true for the generation of anisotropic meshes. 
In this section we give a first numerical assessment on polytopal anisotropic mesh refinement. We propose a bisection approach that does not rely on any initially prescribed direction and which is applicable in two- and three-dimensions. Classical bisection approaches for triangular and tetrahedral meshes do not share this versatility and they have to be combined with additional strategies like edge swapping, node removal and local node movement, see~\cite{Schneider2013}.

Starting from the local interpolation error estimate in Theorem~\ref{th:ClementInterpolationEstimates}, we obtain the global version
\[
	\|v-\mathfrak I_Cv\|_{L_2(\Omega)}\leq c\left(\sum_{K\in\mathcal K_h}\|A_K^{-\top}\nabla v\|_{L_2(K)}^2\right)^{1/2}
\]
by exploiting Remark~\ref{rem:NeighbouringElements} and Proposition~\ref{prop:BoundedNumOfNodesVSElems}. As in the derivation of Proposition~\ref{prop:ClementInterpolationEstimates}, we easily see that
\[
  \eta = \sqrt{\sum_{K\in\mathcal K_h}\eta_K}
  \qquad\mbox{with}\qquad
  \eta_K = \alpha_K^{-2}
      \sum_{j=1}^d \lambda_{K,j}\,\vec u_{K,j}^\top\, G_K^*(v)\, \vec u_{K,j}
\]
and
\[
  G_K^*(v)
  = \left(
      \int_{K} \frac{\partial v}{\partial x_i}\frac{\partial v}{\partial x_j}\,d\vec x 
    \right)_{i,j=1}^d\in\mathbb R^{d\times d},
  \quad
  \vec x = (x_1,\ldots,x_d)^\top.
\]
is a good error measure and the local values~$\eta_K$ may serve as error indicators over the polytopal elements. This estimate also remains meaningful on isotropic polytopal meshes, cf.\ Remark~\ref{rem:CaseIsotropic}. In the case that $v\in H^1(\Omega)$ and its derivatives are known, we can thus apply the following adaptive mesh refinement algorithm:
\begin{enumerate}
 \item Let $\mathcal{K}_0$ be a given initial mesh and $\ell=0$.
 \item\label{item:alg_errindicators} Compute the error indicators $\eta_K$ and $\eta$ with the knowledge of the exact function~$v$ and its derivatives.
 \item\label{item:alg_mark} Mark all elements~$K$ for refinement which fulfil $\eta_K > 0.9\eta^2/|\mathcal K_\ell|$, where $|\mathcal K_\ell|$ is the number of elements in the current mesh.
 \item\label{item:alg_refine} Refine the marked elements as described below in order to obtain a refined mesh $\mathcal K_{\ell+1}$.
 \item Go to~\ref{item:alg_errindicators}.
\end{enumerate}
In step~\ref{item:alg_mark}, we have chosen a equidistribution strategy which marks all elements for refinement whose error indicator is larger than the mean value. The factor $0.9$ has been chosen for stabilizing reasons in the computations when the error is almost uniformly distributed. For the refinement in step~\ref{item:alg_refine}, we have a closer look at the first term in the sum of~$\eta_K$, which reads
\[
  \lambda_{K,1}\;\frac{\vec u_{K,1}^\top\, G_K^*(v)\, \vec u_{K,1}}{\vec u_{K,1}^\top\, \vec u_{K,1}},
\]
because of $|\vec u_{K,1}|=1$. Since $\lambda_{K,1}\gg\lambda_{K,d}$ for anisotropic elements, the refinement process should try to minimize the quotient such that the whole term does not dominate the error over~$K$. Obviously, we are dealing here with the Rayleigh quotient, which is minimal if $\vec u_{K,1}$ is the eigenvector to the smallest eigenvalue of $G_K^*(v)$. As consequence, the longest stretching of the polytopal element~$K$ should be aligned with the direction
of this eigenvector.
In order to achieve the correct alignment for the next refined mesh, we may bisect the polytopal element orthogonal to the eigenvector which belongs to the largest eigenvalue of $G_K^*(v)$. Thus, we propose the following refinement strategies:
\begin{description}
 \item[ISOTROPIC] The elements are bisected as introduced in~\cite{Weisser2011}, i.e., they are split orthogonal to the eigenvector corresponding to the largest eigenvalue of $M_\mathrm{Cov}(K)$.
 \item[ANISOTROPIC] In order to respect the anisotropic nature of~$v$, we split the elements orthogonal to the eigenvector corresponding to the largest eigenvalue of $G_K^*(u)$.
\end{description}

For the numerical experiments we consider $\Omega=(0,1)^2$ and the function
\begin{equation}\label{eq:Tanh}
 v(x_1,x_2) = \tanh(60x_2) - \tanh(60(x_1-x_2)-30),
\end{equation}
taken from~\cite{HuangKamenskiLang2010}, which has two sharp layers: one along the $x_1$-axis and one along the line given by $x_2= x_1-0.5$. The function as well as the initial mesh is depicted in Fig.~\ref{fig:Tanh}.
\begin{figure}[tbp]
  \includegraphics[trim=3cm 0cm 4.5cm 0cm, width=0.3\textwidth, clip]{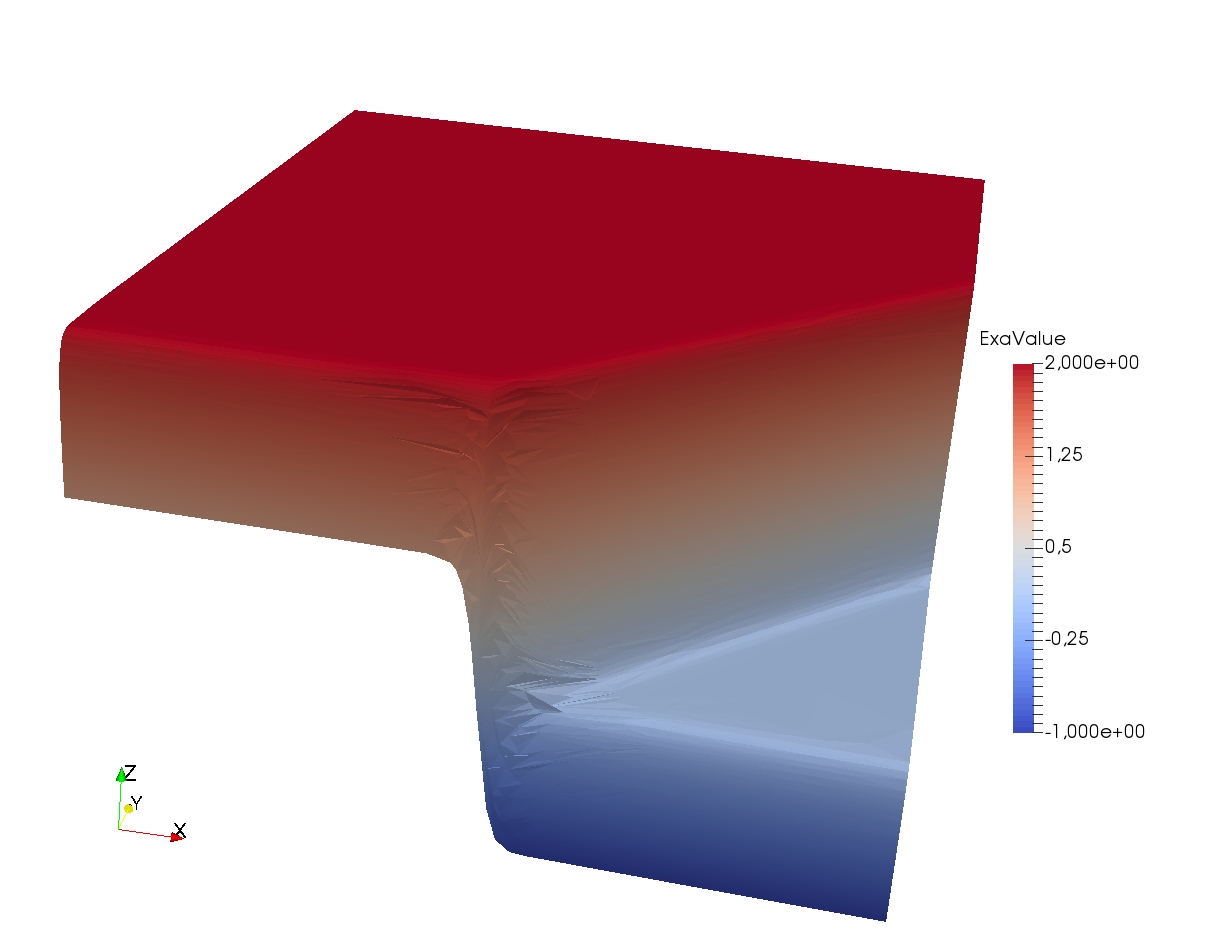}\hfil
  \includegraphics[trim=4.1cm 7.2cm 3.6cm 9.4cm, width=0.27\textwidth, clip]{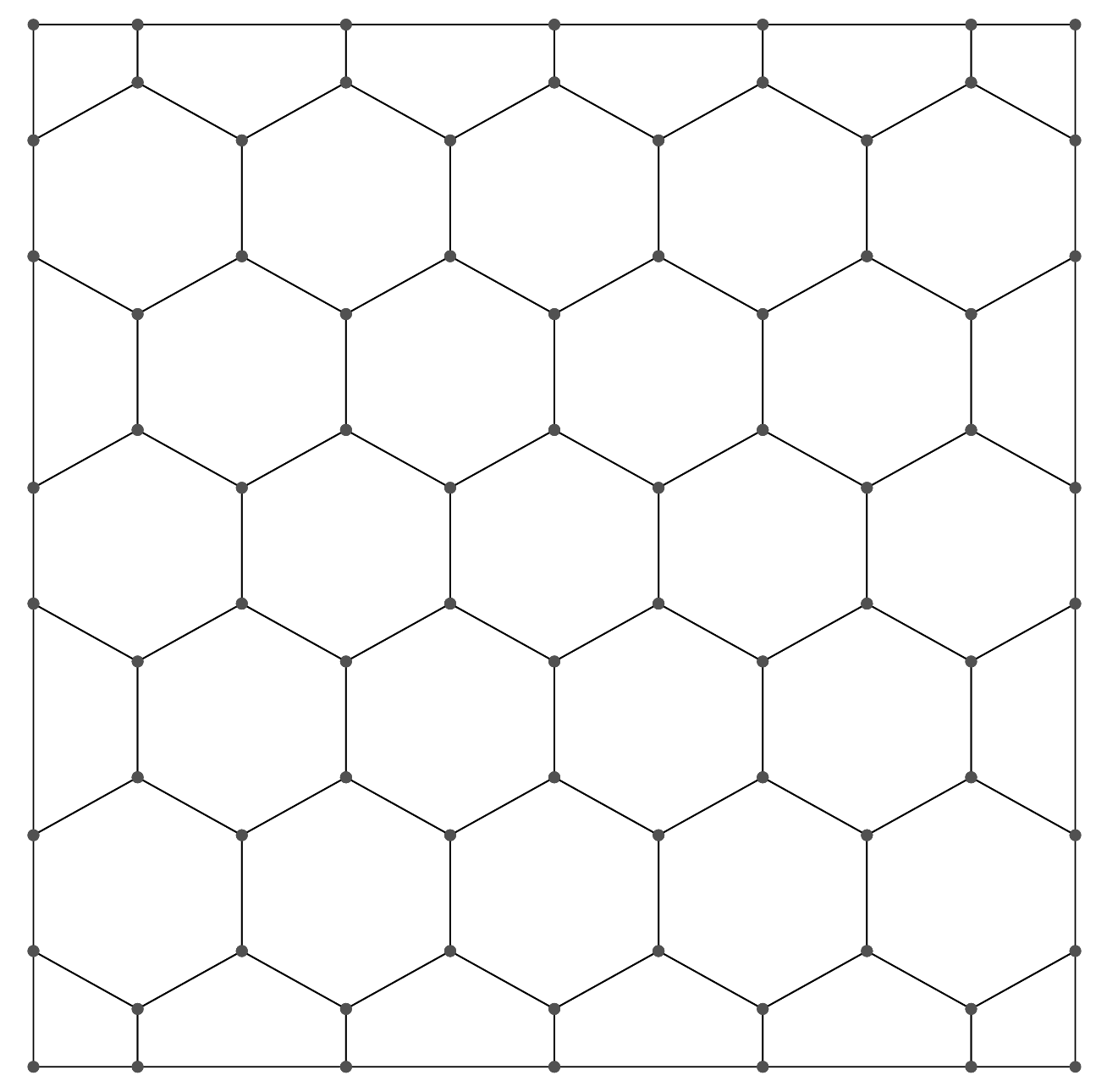}
  \caption{Visualization of function with anisotropic behaviour (left) and initial mesh (right)}
 \label{fig:Tanh}
\end{figure}

\subsection{Test: mesh refinement}
\label{subsec:Test1}
In the first test we generate several sequences of polygonal meshes starting from an initial grid, see Fig.~\ref{fig:Tanh} right. These meshes contain naturally hanging nodes and their element shapes are quite general. First, the initial mesh is refined uniformly, i.e.\ all elements of the discretization are bisected in each refinement step. Here, the ISOTROPIC strategy is performed for the bisection. The mesh after $6$ refinements as well as a zoom-in is depicted in Fig.~\ref{fig:UNIFORM}. The uniform refinement clearly generates a lot of elements in regions where the function~\eqref{eq:Tanh} is flat and where only a few elements would be sufficient for the approximation.
\begin{figure}[tbp]
 \includegraphics[trim=2.5cm 6.0cm 3.1cm 8.5cm, width=0.45\textwidth, clip]{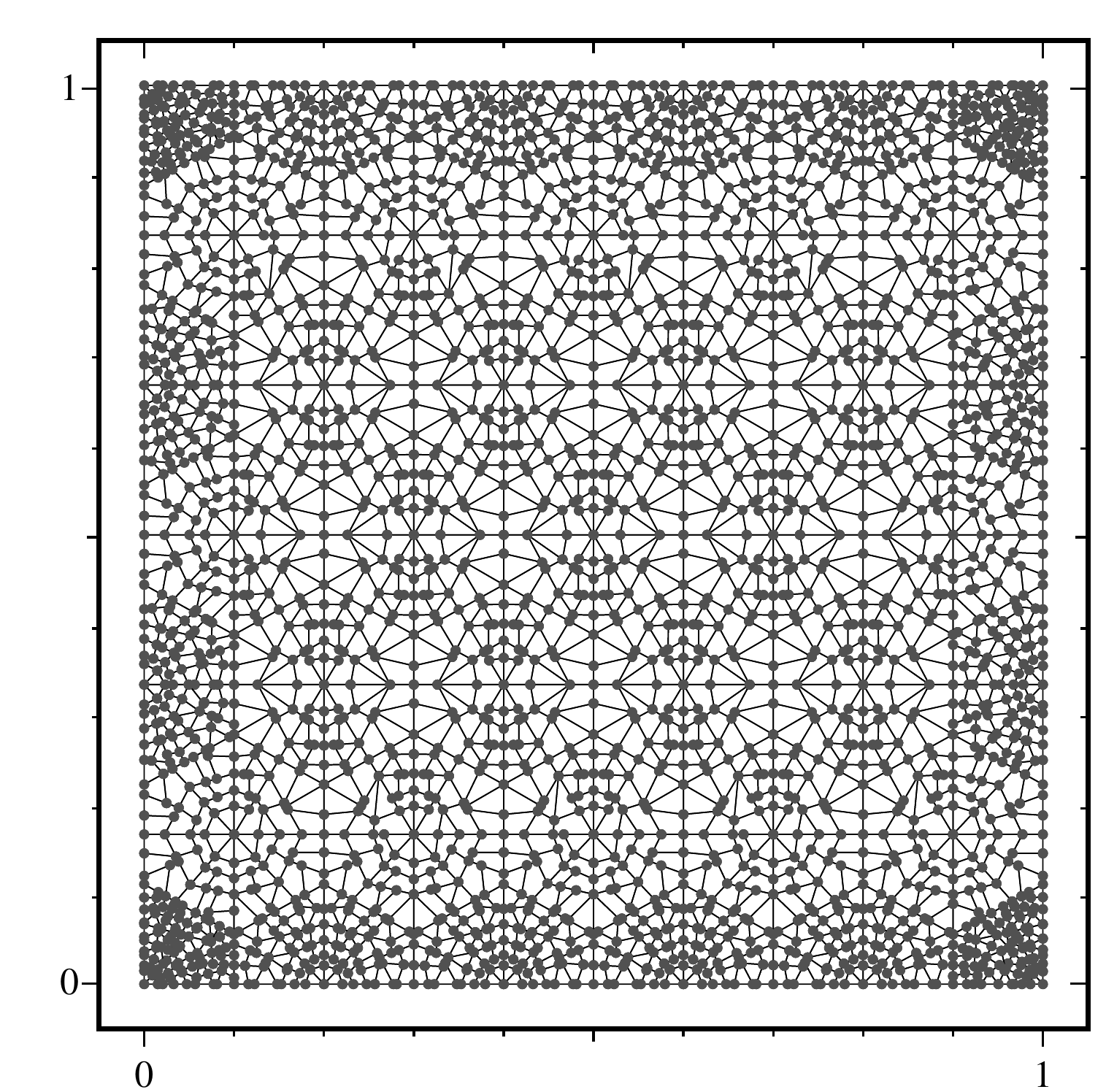}\hfil
 \includegraphics[trim=2.5cm 6.0cm 3.1cm 8.5cm, width=0.45\textwidth, clip]{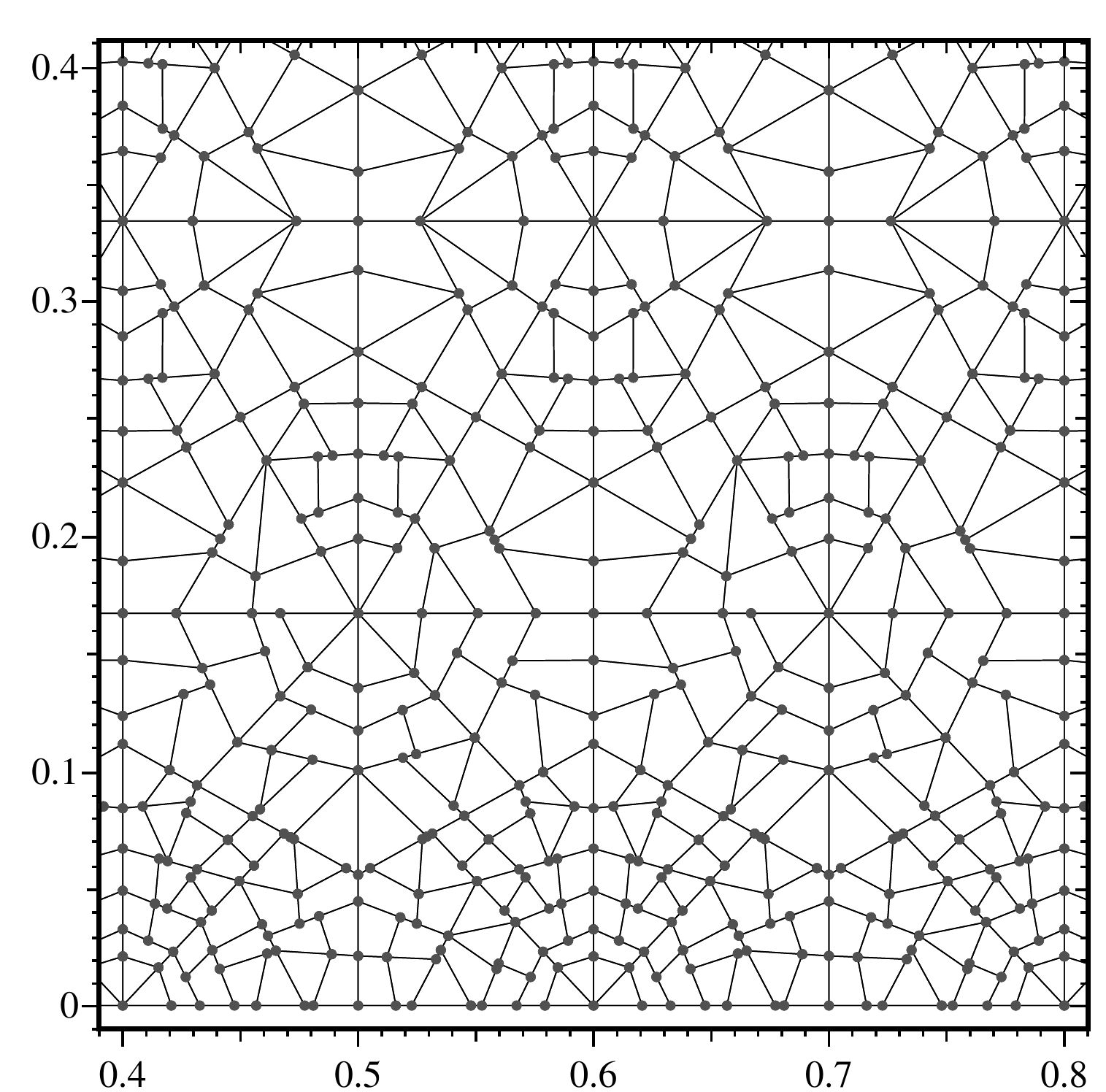}
 \caption{Mesh after $6$ uniform refinement steps using the ISOTROPIC  strategy and zoom-in}
 \label{fig:UNIFORM}
\end{figure}

Next, we perform the adaptive refinement algorithm as described above for the different bisection strategies. The generated meshes after $6$ refinement steps are visualized in the Figs.~\ref{fig:ISOTROPIC} and~\ref{fig:ANISOTROPIC} 
together with a zoom-in of the region where the two layers of the function~\eqref{eq:Tanh} meet. 
Both strategies detect the layers and adapt the refinement to the underlying function. The adaptive strategies clearly outperform the uniform refinement with respect to the number of nodes which are needed to resolve the layers. 
Whereas the ISOTROPIC strategy in Fig.~\ref{fig:ISOTROPIC} keeps the aspect ratio of the polygonal elements bounded, the ANISOTROPIC bisection produces highly anisotropic elements, see Fig.~\ref{fig:ANISOTROPIC}. These anisotropic elements coincide with the layers of the function very well. 

Finally, we compare the error measure~$\eta$ for the different strategies. This value is given with respect to the number of degrees of freedom, which coincides with the number of nodes, in a double logarithmic plot in Fig.~\ref{fig:convergence_etaKnown}. The error measure decreases most rapidly for the ANISOTROPIC strategy and consequently these meshes are most appropriate for the approximation of the function~\eqref{eq:Tanh}. The slope corresponds to quadratic convergence in finite element analysis.
\begin{figure}[tbp]
 \includegraphics[trim=2.5cm 6.0cm 3.1cm 8.5cm, width=0.45\textwidth, clip]{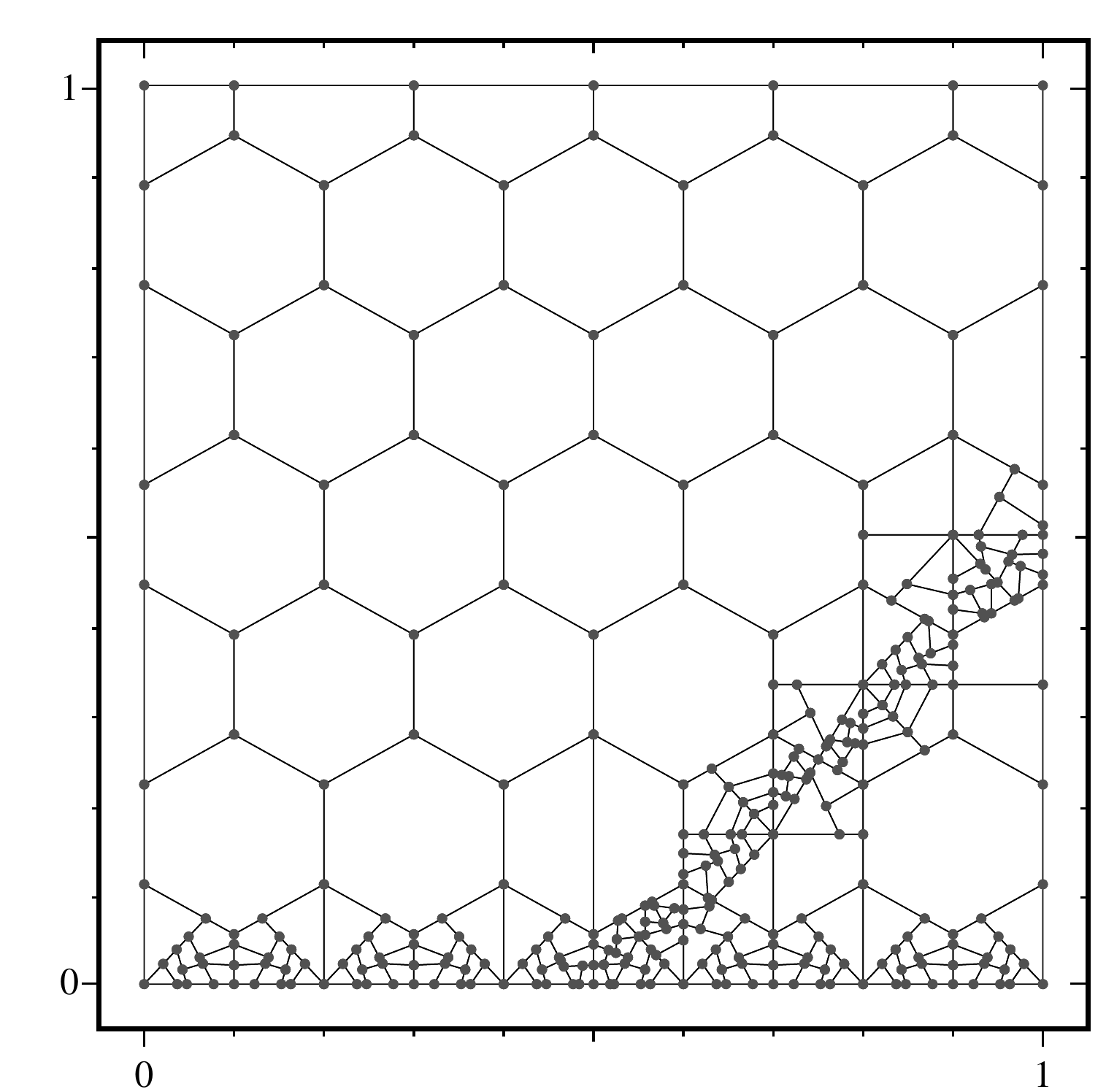}\hfil
 \includegraphics[trim=2.5cm 6.0cm 3.1cm 8.5cm, width=0.45\textwidth, clip]{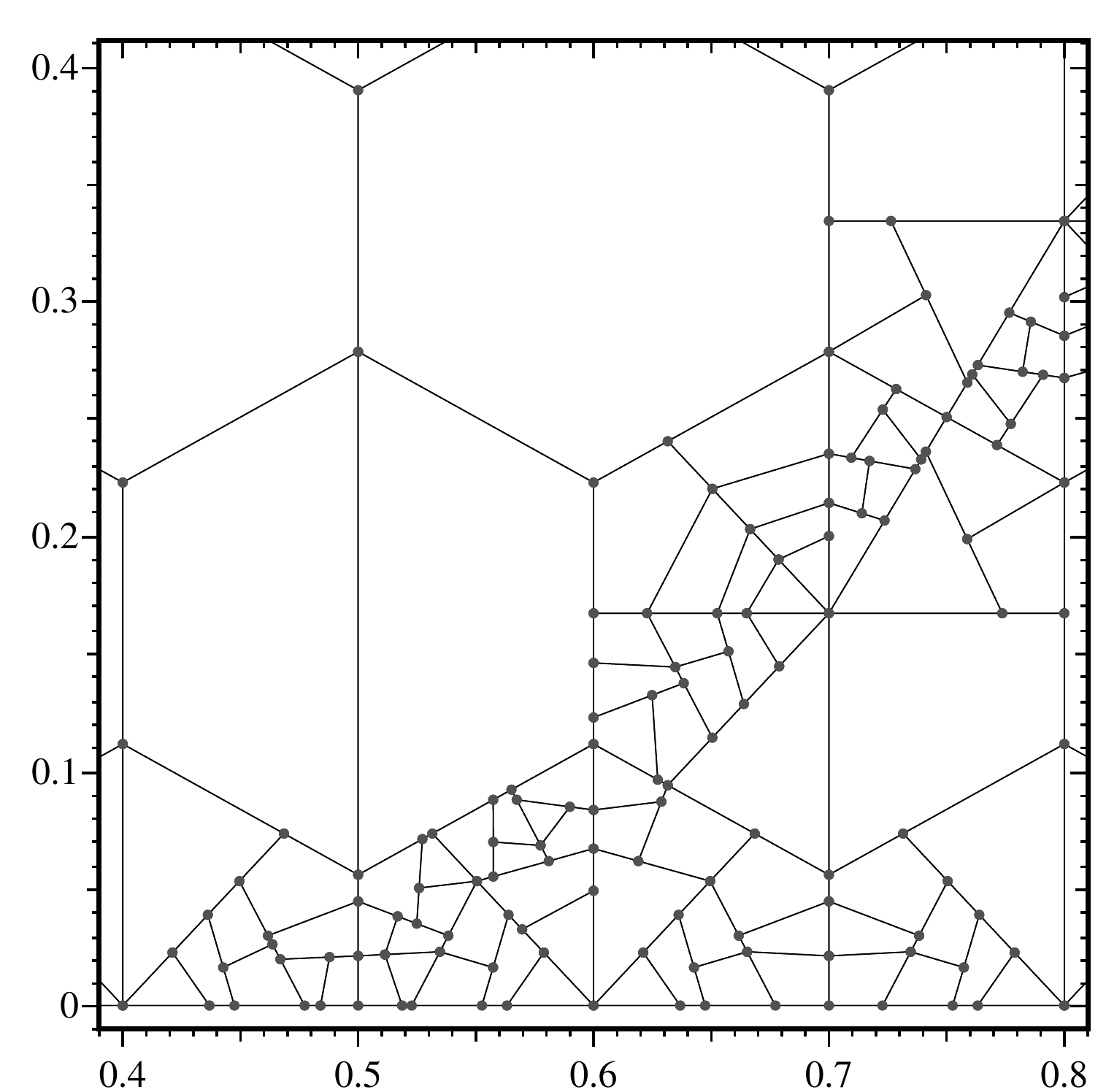}
 \caption{Mesh after $6$ adaptive refinement steps for the ISOTROPIC strategy and zoom-in}
 \label{fig:ISOTROPIC}
\end{figure}
\begin{figure}[tbp]
 \includegraphics[trim=2.5cm 6.0cm 3.1cm 8.5cm, width=0.45\textwidth, clip]{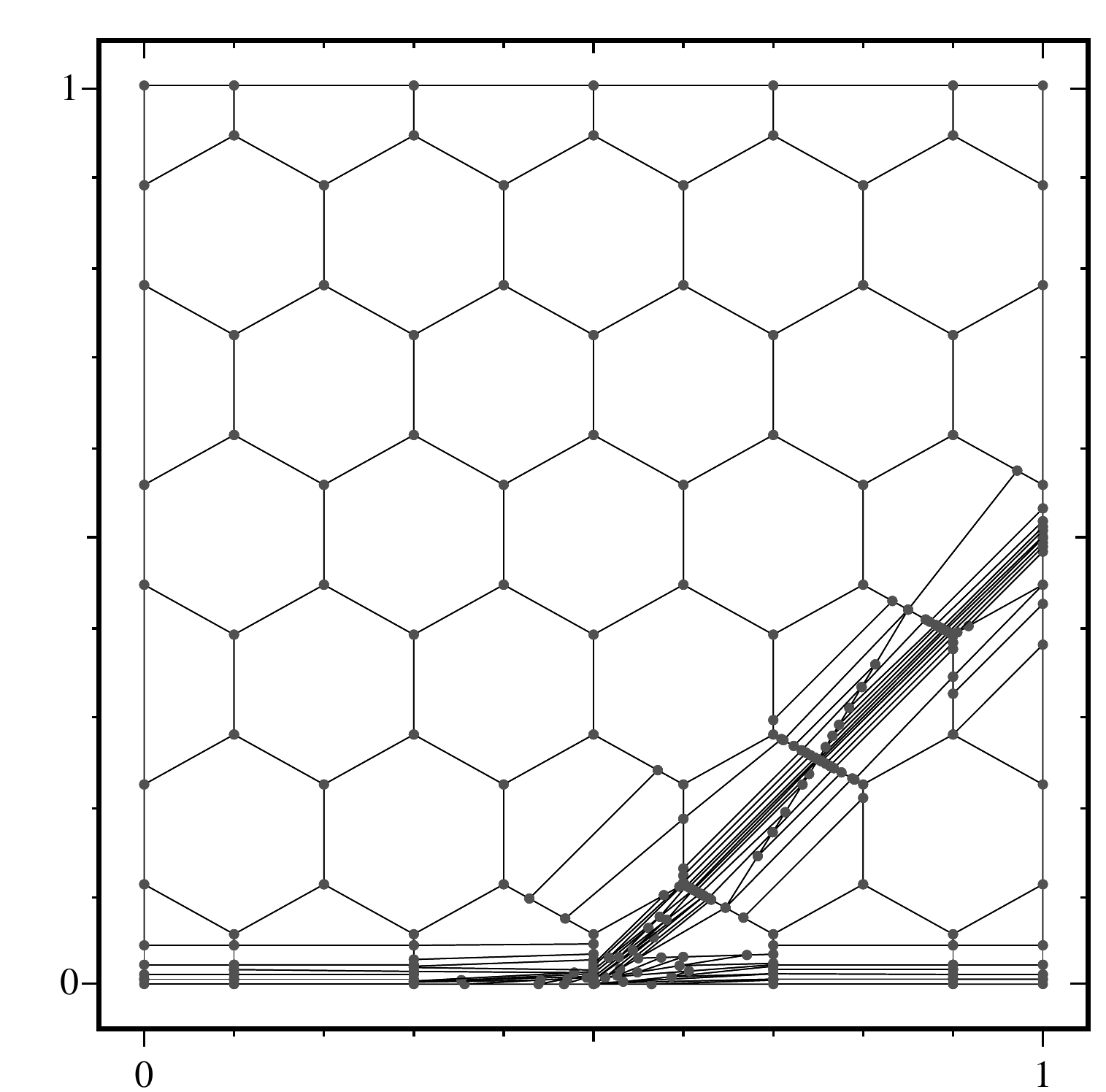}\hfil
 \includegraphics[trim=2.5cm 6.0cm 3.1cm 8.5cm, width=0.45\textwidth, clip]{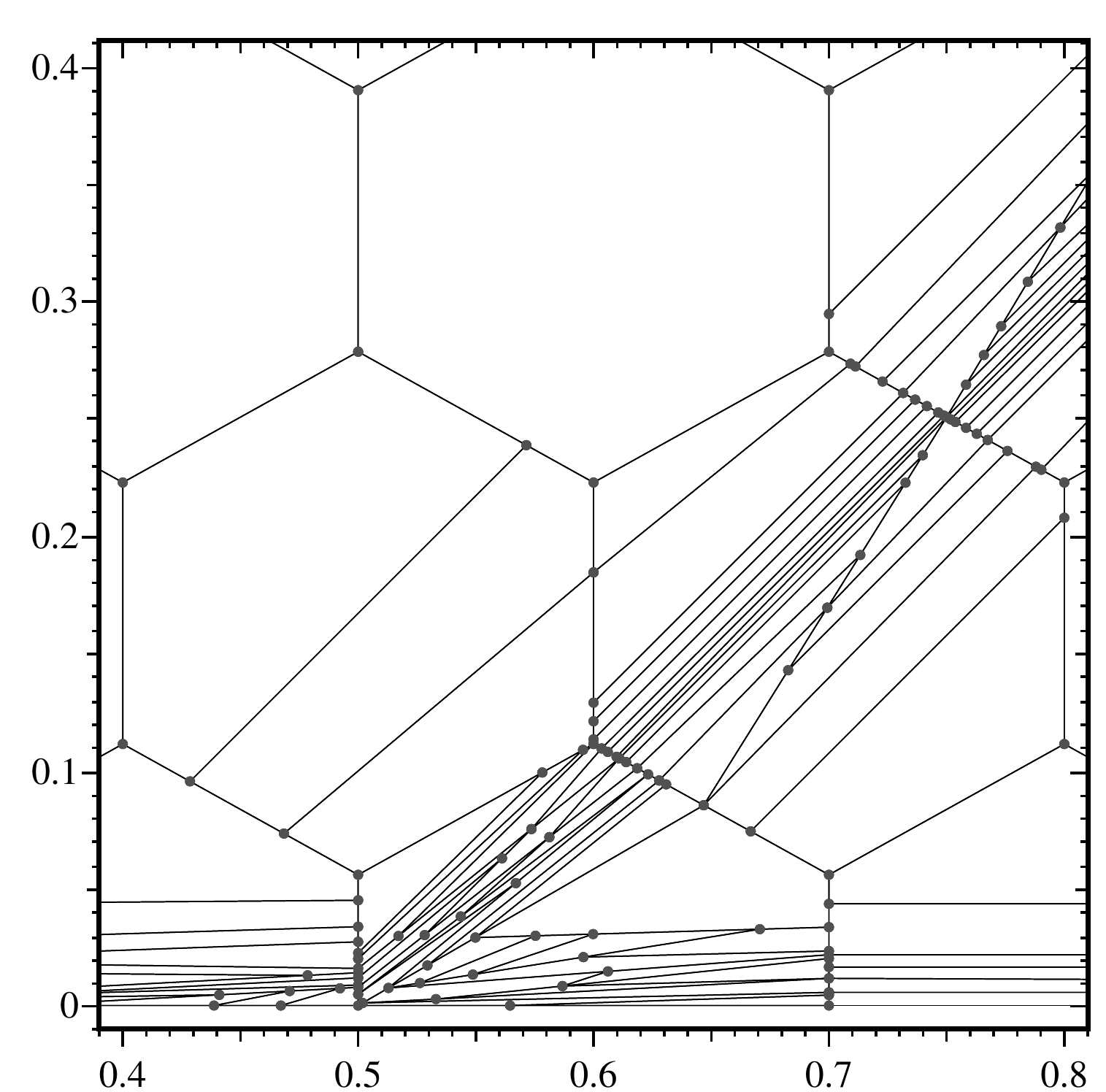}
 \caption{Mesh after $6$ adaptive refinement steps for the ANISOTROPIC strategy and zoom-in}
 \label{fig:ANISOTROPIC}
\end{figure}
\begin{figure}[tbp]
 \scalebox{0.9}{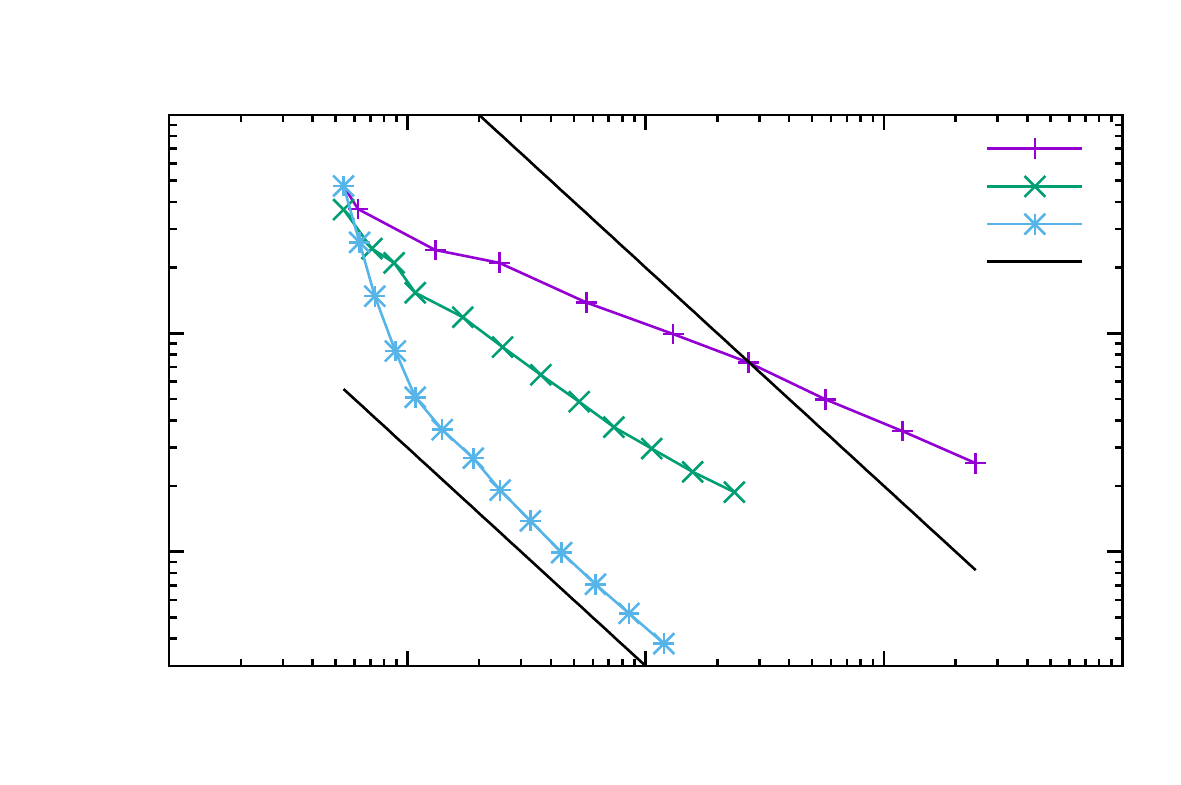}
 \caption{Convergence graph of the anisotropic error measure~$\eta$ with respect to the number of degrees of freedom for the different refinement strategies}
 \label{fig:convergence_etaKnown}
\end{figure}

\subsection{Test: mesh properties}
We analyse the meshes more carefully. For this purpose we pick the 13th mesh of the sequence generated with the ISOTROPIC and the ANISOTROPIC refinement strategy. In Sec.~\ref{subsec:Characterisation}, we have introduced the ratio $\lambda_{K,1}/\lambda_{K,2}$ for the characterisation of the anisotropy of an element. In Fig.~\ref{fig:MeshInfo_lambda}, we give this ratio with respect to the element ids for the two chosen meshes. For the ISOTROPIC refined mesh the ratio is clearly bounded by~$10$ and therefore the mesh consists of isotropic elements according to our characterisation. In the ANISOTROPIC refined mesh, however, the ratio varies in a large interval. The mesh consists of several isotropic elements, but there are mainly anisotropic polygons. The ratio of the most anisotropic elements exceeds~$10^5$ in this example.
\begin{figure}[tbp]
  \scalebox{0.9}{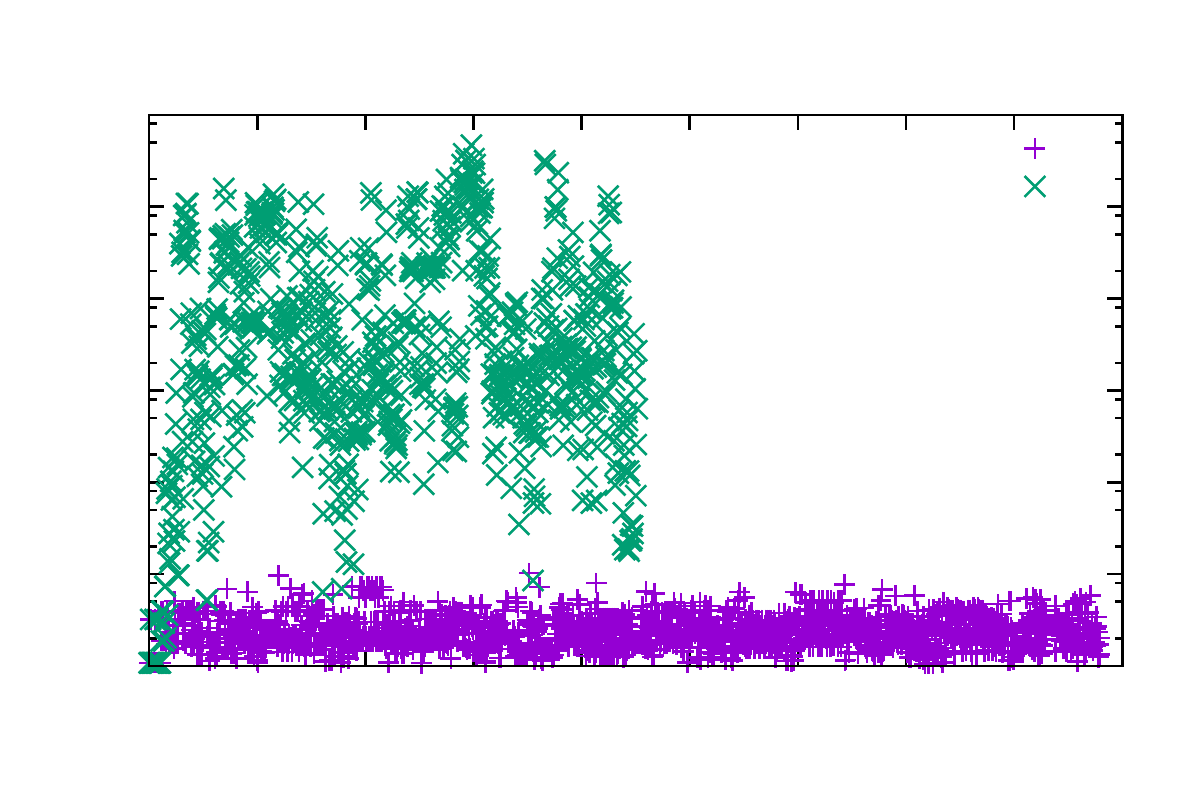}
  \caption{Quotient $\lambda_{K,1}/\lambda_{K,2}$ for all elements in the 13th mesh of the sequence with ISOTROPIC and ANISOTROPIC refinement}
 \label{fig:MeshInfo_lambda}
\end{figure}

Next we address the scaling parameter~$\alpha_K$ in these meshes. In the comparison of the derived estimates with those of Formaggia and Perotto~\cite{FormaggiaPerotto2003}, it has been assumed that $\alpha_K\sim 1$. In Fig.~\ref{fig:MeshInfo_alpha}, we present a histogram for the distribution of~$\alpha_K$ in the two selected meshes. As expected the values stay bounded for the ISOTROPIC refined mesh. Furthermore, $\alpha_K$ stays in the same range for the ANISOTROPIC refinement. In our example, all values lie in the interval $(0.28,0.32)$ although we are dealing with elements of quite different aspect ratios, cf. Fig.~\ref{fig:MeshInfo_lambda}.
\begin{figure}[tbp]
  \scalebox{0.9}{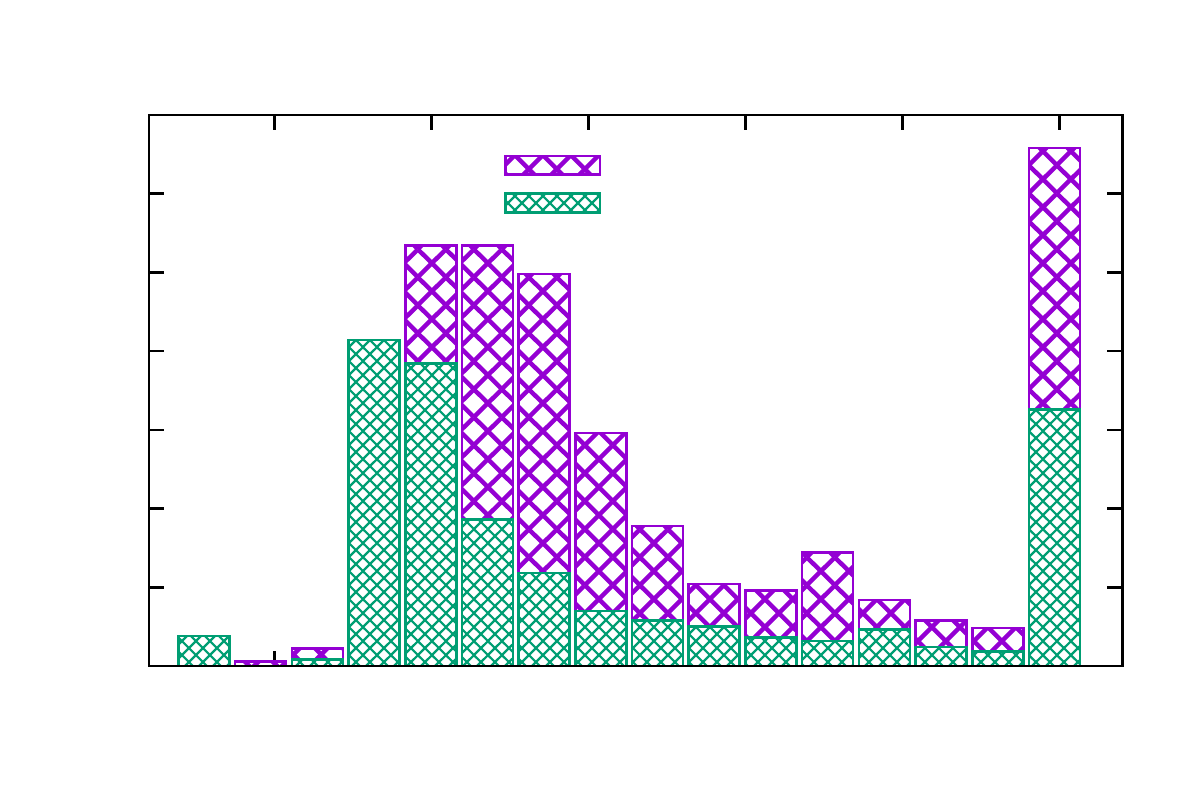}
  \caption{Histogram for the distribution of $\alpha_K$ for the 13th mesh in the sequence with ISOTROPIC and ANISOTROPIC refinement}
 \label{fig:MeshInfo_alpha}
\end{figure}

\subsection{Test: interpolation error}
In the last test we apply the pointwise interpolation operator~\eqref{eq:PWInterpolationOp} to the function~\eqref{eq:Tanh} over the meshes generated in Sec.~\ref{subsec:Test1} and study numerically the convergence. The computations are done with a BEM-based FEM implementation written in C. For more details we refer the interested reader to~\cite{RjasanowWeisser2012}. Here, the implicitly defined basis functions~$\psi_i$ are treated locally by means of boundary element methods (BEM). The implementation uses the coarsest possible BEM discretization and it is not yet adapted to handle anisotropic elements. Due to the existence of a representation formula, it is possible to evaluate~$\psi_i$ inside the elements and thus to approximate, e.g., the $L_2$-norm with the help of numerical quadrature over polygonal elements.

\begin{figure}[tbp]
 \scalebox{0.9}{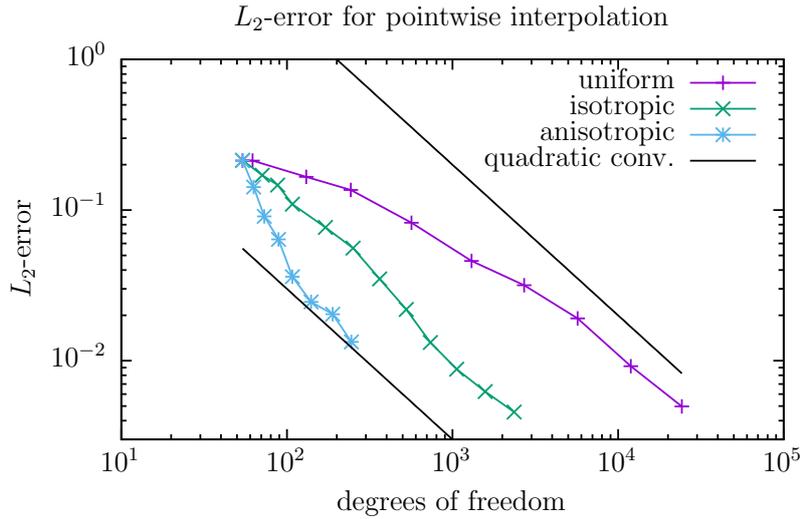}
 \caption{Convergence graph of the $L_2$-error with respect to the number of degrees of freedom for the different refinement strategies}
 \label{fig:convergence_L2}
\end{figure}
The convergence of the interpolation is studied numerically for the different sequences of meshes. We consider the interpolation error in the $L_2$-norm. In Fig.~\ref{fig:convergence_L2}, we give $\|v-\mathfrak I_\mathrm{pw}v\|_{L_2(\Omega)}$ with respect to the number of degrees of freedom in a double logarithmic plot. 
Since $v\in H^2(\Omega)$ in this experiment, we expect quadratic convergence with respect to the mesh size on the sequence of uniform refined meshes. This convergence rate corresponds to a slope of one in the double logarithmic plot in two-dimensions. 
In Fig.~\ref{fig:convergence_L2}, we observe that the uniform refinement reaches indeed quadratic convergence after a pre-asymptotic regime. The optimal rate of convergence is achieved as soon as the layers are resolved in the mesh. On the adaptive generated meshes, however, the interpolation error converges with optimal rates from the beginning. 
We can even recognize in Fig.~\ref{fig:convergence_L2} that the ANISOTROPIC refined meshes outperforms the others. The layers are captured within a few refinement steps and therefore the error reduces faster than for the ISOTROPIC refined meshes before it reaches the optimal convergence rate.

Let us compare the seventh meshes in the sequences which are obtained after six refinements and which are visualized in Figs.~\ref{fig:UNIFORM}--\ref{fig:ANISOTROPIC}. 
For the uniform refined mesh we have $2709$ nodes and it is $\|v-\mathfrak I_\mathrm{pw}v\|_{L_2(\Omega)} \approx 3.17\cdot 10^{-2}$. 
The adaptive refined mesh using ISOTROPIC bisection contains only $363$ nodes but yields a comparable error $\|v-\mathfrak I_\mathrm{pw}v\|_{L_2(\Omega)} \approx 3.49\cdot 10^{-2}$. 
The most accurate approximation is achieved on the ANISOTROPIC refined mesh with $\|v-\mathfrak I_\mathrm{pw}v\|_{L_2(\Omega)} \approx 2.04\cdot 10^{-2}$ and only $189$ nodes. A comparable interpolation error to the other refinement strategies is obtained on the fifth mesh of the sequence of ANISOTROPIC refined meshes. This mesh consists of $108$ nodes only.

\section{Conclusion}
\label{sec:Conclusion}
As seen in the previous section, polygonal elements allow for highly anisotropic meshes which are aligned to layers of the approximated function. Since hanging nodes are naturally included in the discretization, the refinements in the mesh are kept very local. These two properties result in the fact that approximations on anisotropic polytopal meshes are as accurate as on uniform and adaptive isotropic meshes but involving much less degrees of freedom. In consequence, the computational cost is reduced and therefore the efficiency increases.

The results for the derived interpolation and quasi-interpolation operators and their a priori error estimates are in accordance with previous works on classical element shapes. However, the new findings are applicable on much more general anisotropic polytopal meshes. In future projects we aim to apply our results for adaptive finite element strategies involving a posteriori error estimates on polytopal meshes for boundary value problems with highly anisotropic solutions.

\section*{Acknowledgement}
The author would like to thank Paola Antonietti and Marco Verani for their valuable comments during a research stay at the  Laboratory for Modeling and Scientific Computing MOX, Dipartimento di Matematica, Politecnico di Milano, Italy.

\end{document}